\documentclass{birkjour}

\usepackage{amsfonts}
\usepackage{commath}
\usepackage{amsmath}
\usepackage{mathrsfs}
\usepackage{bm}
\usepackage{amsopn}
\usepackage{cleveref}
\usepackage{xcolor}
\usepackage[most]{tcolorbox}
\newtcolorbox{greyFrameTransmission}{
  nobeforeafter,
  colframe=black!10!white,
  colback=white,
  top=2mm, bottom=3 mm, left=0mm, right=0mm,
  arc=0mm,
  fonttitle=\bfseries\color{black!70!black},
  title= \center Transmission Problem
}
\newtcolorbox{greyFrameCoupledProblem}{
  colframe=black!10!white,
   colback =white,
  top=2mm, bottom=3 mm, left=0mm, right=0mm,
  arc=0mm,
  fonttitle=\bfseries\color{black!70!black},
  title= \center Symmetrically Coupled Problem
}
\newtcolorbox{greyFrameSymmetry}{
  colframe=black!20!white,
  top=2mm, bottom=3 mm, left=0mm, right=0mm,
  arc=0mm,
  fonttitle=\bfseries\color{black!70!black},
  title= \center Symmetry of Calder\'on Projector
}

\DeclareMathOperator{\diag}{diag}
\newcommand\ip[2]{\langle #1, #2 \rangle}
\newcommand\bra[1]{\left(#1\right)}

\newcommand\uext{\mathbf{U}^{\text{ext}}}
\newcommand\boldcurl{\mathbf{curl}\,}
\newcommand{\id}{\mathrm{Id}}
\newcount\colveccount
\newcommand*\colvec[1]{
  \global\colveccount#1
  \begin{pmatrix}
    \colvecnext
  }
  \def\colvecnext#1{
    #1
    \global\advance\colveccount-1
    \ifnum\colveccount>0
    \\
    \expandafter\colvecnext
    \else
  \end{pmatrix}
  \fi
}

\makeatletter
\DeclareFontFamily{OMX}{MnSymbolE}{}
\DeclareSymbolFont{MnLargeSymbols}{OMX}{MnSymbolE}{m}{n}
\SetSymbolFont{MnLargeSymbols}{bold}{OMX}{MnSymbolE}{b}{n}
\DeclareFontShape{OMX}{MnSymbolE}{m}{n}{
  <-6>  MnSymbolE5
  <6-7>  MnSymbolE6
  <7-8>  MnSymbolE7
  <8-9>  MnSymbolE8
  <9-10> MnSymbolE9
  <10-12> MnSymbolE10
  <12->   MnSymbolE12
}{}
\DeclareFontShape{OMX}{MnSymbolE}{b}{n}{
  <-6>  MnSymbolE-Bold5
  <6-7>  MnSymbolE-Bold6
  <7-8>  MnSymbolE-Bold7
  <8-9>  MnSymbolE-Bold8
  <9-10> MnSymbolE-Bold9
  <10-12> MnSymbolE-Bold10
  <12->   MnSymbolE-Bold12
}{}

\let\llangle\@undefined
\let\rrangle\@undefined
\DeclareMathDelimiter{\llangle}{\mathopen}%
{MnLargeSymbols}{'164}{MnLargeSymbols}{'164}
\DeclareMathDelimiter{\rrangle}{\mathclose}%
{MnLargeSymbols}{'171}{MnLargeSymbols}{'171}
\makeatother

\newcommand{\hly}[1]{%
  \colorbox{yellow!30}{$\displaystyle#1$}}

\newcommand{\hlo}[1]{%
  \colorbox{orange!10}{$\displaystyle#1$}}
\newcommand{\hlb}[1]{%
  \colorbox{cyan!10}{$\displaystyle#1$}}

\newcommand{\hlm}[1]{%
  \colorbox{magenta!10}{$\displaystyle#1$}}

\newtcolorbox{greyFrame}{
  nobeforeafter,
  colframe=white,
   colback =red!20!white,
}
\newtheorem{theorem}{Theorem}[section]
\newtheorem{corollary}[theorem]{Corollary}
\newtheorem{lemma}[theorem]{Lemma}
\newtheorem{proposition}[theorem]{Proposition}
\theoremstyle{remark}
\newtheorem*{warning}{Warning}
\newtheorem{remark}{Remark}

\begin{document}

%
%
%
%
%
%
%
%
%

\title[Symmetric Coupling for Hodge-Helmholtz Operators]
{Coupled Domain-Boundary Variational\\Formulations for Hodge--Helmholtz \\Operators}

\author[Schulz]{Erick Schulz}

\address{%
  ETH Z\"urich, SAM\\
  HG G 58.3\\
  CH 8092 Z\"urich\\
  Switzerland}

\email{erick.schulz@sam.math.ethz.ch}

\thanks{The work of Erick Schulz was supported by SNF as part of the grant 200021\_184848/1}
\author[Hiptmair]{Ralf Hiptmair}
\address{%
  ETH Z\"urich, SAM\\
  HG G 58.3\\
  CH 8092 Z\"urich\\
  Switzerland}
\email{ralf.hiptmair@sam.math.ethz.ch}


\date{June 7, 2020}

\begin{abstract}
  We couple the mixed variational problem for the generalized Hodge-Helmholtz or
  Hodge-Laplace equation posed on a bounded 3D Lipschitz domain with the
  first-kind boundary integral equations arising from the latter when constant
  coefficients are assumed in the unbounded complement. Recently developed Calder\'on
  projectors for the relevant boundary integral operators are used to perform a symmetric
  coupling. We prove stability of the coupled problem away from resonant frequencies by establishing a generalized G{\aa}rding inequality (T-coercivity). The resulting system of equations describes the scattering of monochromatic electromagnetic waves at a bounded inhomogeneous isotropic body possibly having a ``rough" surface. The low-frequency robustness of the potential formulation of Maxwell's equations makes this model a promising starting point for Galerkin discretization.
\end{abstract}

\maketitle
\section{Introduction}\label{sec: Introduction}
Let $\Omega_s\subset\mathbb{R}^3$ be a bounded Lipschitz domain \cite[Def. 2.1]{steinbach2007numerical} representing a region of space occupied by a dielectric object, the scatterer, with spatially varying material properties. The scalar material coefficients are assumed to be bounded, i.e. $\mu,\epsilon\in L^{\infty}(\mathbb{R}^3)$. In a non-dissipative medium, the functions $\mu$ and $\epsilon$ are real-valued and uniformly positive. Dissipative effects are captured by allowing the coefficients to have non-negative imaginary parts \cite[Sec. 1.1.3]{assous2018mathematical}. We follow \cite{hazard1996solution} and suppose that
\begin{align*}
&0<\mu_{\text{min}}\leq\mathfrak{Re}(\mu)\leq\mu_{\text{max}},& &0\leq\mathfrak{Im}(\mu) ,\\
&0<\epsilon_{\text{min}}\leq\mathfrak{Re}(\epsilon)\leq\epsilon_{\text{max}},& &0\leq\mathfrak{Im}(\epsilon),\\
&0 \leq \mathfrak{Re}(\kappa^2), & &0\leq\mathfrak{Im}(\kappa^2).     
\end{align*} We assume for simplicity that $\Omega_s$ has trivial cohomology, in other words that its first and second Betti numbers are zero \cite[Sec. 4.4]{arnold2018}. Qualitatively, this means that it doesn't feature  handles nor interior voids: it is homeomorphic to a ball. 

\begin{remark}
The hypothesis that the second Betti number is zero is only used to prove injectivity of the coupling problem for Hodge--Laplace operators. It can be dropped without any changes to the following development for couplings involving the Hodge--Helmholtz operator (\emph{non-static} electromagnetic transmission problems). The hypothesis that the first Betti number is zero is used in \Cref{sec: compactness and coercivity} to guarantee the existence of a certain ``scalar potential lifting" that greatly simplifies the Fredholm arguments.
\end{remark}

Inside this possibly inhomogeneous isotropic physical body, the potential
formulation of Maxwell's equations in frequency domain driven by a source current
$\mathbf{J}\in\mathbf{L}^{2}(\Omega_s)$ with angular frequency $\omega>0$ reads \cite{chew2014vector}
\begin{subequations}
  \begin{align}
    \mathbf{curl}\,\bra{\mu^{-1}(\mathbf{x})\,\mathbf{curl}\,\mathbf{U}} +i\omega\epsilon(\mathbf{x}) \nabla V - \omega^2 \epsilon(\mathbf{x})\,\mathbf{U} &= \mathbf{J},\\
    \text{div}\,\bra{\epsilon(\mathbf{x})\mathbf{U}}+i\omega V &=0,\label{eq: Lorentz gauge}
  \end{align}
  \end{subequations}
where the Lorentz gauge \eqref{eq: Lorentz gauge} relates the scalar potential $V$ to the vector potential
$\mathbf{U}$. Elimination of $V$ using this relation leads to
the Hodge--Helmholtz equation
\begin{equation}
  \mathbf{curl}\,\bra{\mu^{-1}(\mathbf{x})\,\mathbf{curl}\,\mathbf{U}} - \epsilon(\mathbf{x})\,\nabla\,\text{div}\, \bra{\epsilon(\mathbf{x})\mathbf{U}} - \omega^2\epsilon(\mathbf{x})\,\mathbf{U} = \mathbf{J}. \label{eq: div Maxwell elimination}
\end{equation}

Away from the source current, in the unbounded region
$\Omega':=\mathbb{R}^3 \backslash \overline{\Omega}_s$ outside the scatterer $\Omega_s$, where we
assume a homogeneous material with scalar constant permeability $\mu_0>0$ and dielectric
permittivity $\epsilon_0>0$, equation \eqref{eq: div Maxwell elimination} reduces to
\begin{equation*}
  \Delta_{\eta}\mathbf{U}-\kappa^2\mathbf{U}:=\mathbf{curl}\,\mathbf{curl}\,\mathbf{U} - \eta\,\nabla\,\text{div}\,\mathbf{U} - \kappa^2\mathbf{U} = 0,
\end{equation*}
with constant coefficients $\eta = \mu_0\epsilon_0^2$ and $\kappa^2 = \mu_0\epsilon_0\omega^2$.

For given data
$\mathbf{g}_R\in\mathbf{H}^{-1/2}(\text{div}_{\Gamma})$, $g_n\in H^{-1/2}(\Gamma)$,
$\zeta_D\in H^{1/2}(\Gamma)$ and $\bm{\zeta}_t\in\mathbf{H}^{-1/2}(\text{curl}_{\Gamma},\Gamma)$ on the boundary $\Gamma=\partial\Omega_s$,
we are interested in the following transmission problem, cf. \cite[Sec. 2.1.2]{hazard1996solution}, \cite{chew2014vector}:
\newline\newline
\begin{greyFrameTransmission}
\medskip
\emph{Volume equations}
\smallskip
  \begin{subequations}\label{eq: transmission problem}
    \begin{align}
      \mathbf{curl}\,\bra{\mu^{-1}(\mathbf{x})\,\mathbf{curl}\,\mathbf{U}} -
      \epsilon(\mathbf{x})\,\nabla\,\text{div}\, \bra{\epsilon(\mathbf{x})\,\mathbf{U}} -
      \omega^2\epsilon(\mathbf{x})\,\mathbf{U} = \mathbf{J} \,\text{ in }\Omega_s,
      \label{eq: strong from coupled sys a}\\
      \mathbf{curl}\,\mathbf{curl}\,\uext - \eta\,\nabla\text{div}\,\uext -
      \kappa^2\,\uext = 0\, \text{ in }\Omega',
      \label{eq: strong from coupled sys b}
    \end{align}
  \end{subequations}
\medskip
  \emph{Transmission conditions}
\smallskip
  \begin{subequations}
    \begin{align}
      &\gamma^-_{R,\mu}(\mathbf{U}) =
      \gamma^{+}_R\uext+\mathbf{g}_R,\quad\gamma^{-}_{n,\epsilon}(\mathbf{U}) =
      \gamma^+_n(\mathbf{U}^{\text{ext}})+g_n&&\text{on }\Gamma,
      \label{eq: strong from coupled sys c}\\
      &\gamma^-_{D,\epsilon}(\mathbf{U}) = \eta\, \gamma^{+}_D \uext + \zeta_D,\quad
      \gamma_t^- \mathbf{U}-\gamma_t^+ \mathbf{U} = \bm{\zeta}_t
      &&\text{on }\Gamma.
      \label{eq: strong from coupled sys d} 
    \end{align}                        
  \end{subequations} 
\end{greyFrameTransmission}  
\newline\newline
The traces $\gamma_{\bullet}^\mp$, $\bullet=R$, $D$, $n$, etc., on $\Gamma$ from inside (superscript $-$) and outside (superscript $+$) $\Omega_s$ are defined for a smooth vector-field $\mathbf{U}$ by
  \begin{align*}
    \gamma^-_{R,\mu}(\mathbf{U})&:= -\gamma^-_{\tau}\bra{\mu^{-1}\bra{\mathbf{x}}\,\mathbf{curl}(\mathbf{U})}, & 
    \gamma^+_{R}(\uext)&:= -\gamma^+_{\tau}\bra{\mathbf{curl}\bra{\uext}},\\ 
    \gamma^-_{D,\epsilon}(\mathbf{U})&:=\gamma^-\bra{\text{div}\bra{\epsilon\bra{\mathbf{x}}\,\mathbf{U}}}, & 
    \gamma^+_{D}(\uext)&:=\gamma^+\bra{\text{div}\bra{\uext}},\\
    \gamma^{-}_{n,\epsilon}(\mathbf{U})&:= \gamma^-_n(\epsilon\bra{\mathbf{x}}\,\mathbf{U})& 
    \gamma^\pm_{t}\bra{\mathbf{U}} &:= \mathbf{n}\times\bra{\gamma^\pm_{\tau}\bra{\mathbf{U}}},
  \end{align*}
involving the classical traces
\begin{align*}
  &\gamma\bra{\mathbf{U}} := \mathbf{U}\big\vert_{\Gamma}, 
  && \gamma_n\bra{\mathbf{U}} := \gamma\bra{\mathbf{U}}\cdot \mathbf{n},
  &&\gamma_{\tau}\bra{\mathbf{U}} := \gamma\bra{\mathbf{U}}\times \mathbf{n},
\end{align*}
where $\mathbf{n}\in\mathbf{L}^{\infty}(\Gamma)$ is the essentially bounded unit normal vector field on $\Gamma$ directed toward the exterior of $\Omega_s$ \cite[Thm. 3.1.6]{federer2014geometric}. 

For positive frequencies $\omega>0$, we supplement \eqref{eq: strong from coupled sys
  a}-\eqref{eq: strong from coupled sys d} with the variants of the Silver-M\"{u}ller's
radiation condition imposed at infinity provided in \cite{hazard1996solution}. In the
static case where $\kappa=\omega=0$, we seek a solution in an appropriate weighted Sobolev
space that accounts for decay conditions \cite[Sec. 2.5]{schwarz2006hodge}.

\subsection{Our contributions.}
In the following, we couple the \emph{mixed formulation} of the weak variational problem
associated to \eqref{eq: strong from coupled sys a} with the first-kind boundary integral
equation (BIE) arising from \eqref{eq: strong from coupled sys b} using these recently developed
Calder\'on projectors for the Hodge--Helmholtz and Hodge--Laplace operators. 
The proof of the well-posedness of the coupled
problem relies on T-coercivity (c.f. \cite{ciarlet2012t}) and is given in
\cref{sec:main}. It draws on and integrates several fundamental results of the theory of
first-kind boundary integral operators on Lipschitz domains and of the mathematical
analysis of Maxwell's equations:
\begin{itemize}
\item[$\triangleright$] M.~Costabel's symmetric coupling approach linking volume variational equations with
  BIEs \cite{costabel1987symmetric},
\item[$\triangleright$] T-coercivity for electromagnetic variational problems via Hodge--type
  decompositions \cite{hiptmair2003coupling,claeys2017first},
\item[$\triangleright$] mixed variational formulations of boundary value problems for Hodge--Laplace
  operators \cite{arnold2006finite}.
\end{itemize}
A crucial and surprising discovery is the perfect match of the interface terms naturally
arising from the mixed variational formulation and from the first-kind BIE, see \cref{sec:
  coupled problem}, and in particular \eqref{Calderon coupled problem}, for details.

\section{Preliminaries} 
Let $\Omega\in\{\Omega_s,\Omega'\}$. As usual, $L^2(\Omega)$ and $\mathbf{L}^2(\Omega)$ denote the Hilbert spaces of square
integrable scalar and vector-valued functions defined over $\Omega$. We denote their inner
products using round brackets, e.g. $(\cdot,\cdot)_{\Omega}$. Similarly,  $H^1(\Omega)$ and $\mathbf{H}^1(\Omega)$ refer to the corresponding Sobolev spaces. We write
$C^{\infty}_0(\Omega)$ for the space of smooth compactly supported functions in $\Omega$,
but denote by $\mathscr{D}(\Omega)^3$ the analogous space of vector fields to simplify
notation. The Banach spaces
\begin{align*}
    \mathbf{H}(\text{div}, \Omega) &:=\{\mathbf{U}\in L^2(\Omega) \,\vert\, \text{div}(\mathbf{U})\in L^2(\Omega)\},\\
    \mathbf{H}(\epsilon;\text{div}, \Omega) &:=\{\mathbf{U}\in L^2(\Omega) \,\vert\, \epsilon (\mathbf{x})\,\mathbf{U}\in \mathbf{H}(\text{div}, \Omega)\},\\
    \mathbf{H}(\mathbf{curl}, \Omega) &:=\{\mathbf{U}\in L^2(\Omega) \,\vert\, \boldcurl(\mathbf{U})\in L^2(\Omega)\},\\
    \mathbf{H}\bra{\nabla \text{div},\Omega} &:= \{\mathbf{U}\in \mathbf{H}\bra{\text{div},\Omega} \,\vert\, \text{div}(\mathbf{U})\in H^1(\Omega)\},\\
    \mathbf{H}\bra{\epsilon;\nabla \text{div},\Omega_s} &:= \{\mathbf{U}\in \mathbf{L}^2(\Omega)\,\vert\,\epsilon(\mathbf{x})\,\mathbf{U}\in \mathbf{H}\bra{\nabla \text{div},\Omega} \},\\
    \mathbf{H}(\mathbf{curl}^2,\Omega)&:=\{\mathbf{U}\in \mathbf{H}(\mathbf{curl}, \Omega) \,\vert\, \mathbf{curl}(\mathbf{U})\in \mathbf{H}(\mathbf{curl}, \Omega)\},\\
    \mathbf{H}(\mu^{-1};\mathbf{curl}^2,\Omega)&:=\{\mathbf{U}\in \mathbf{H}(\mathbf{curl}, \Omega) \,\vert\, \mu^{-1}\,\mathbf{curl}(\mathbf{U})\in \mathbf{H}(\mathbf{curl}, \Omega)\},
\end{align*}
equipped with the natural graph norms will be important. The variational space for the primal variational formulation of the classical and
generalized Hodge--Helmholtz/Laplace operator is given by
\begin{align}
  \mathbf{X}(\Delta,\Omega)&:=\mathbf{H}(\mathbf{curl}^2,\Omega)\cap \mathbf{H}\bra{\nabla \text{div},\Omega}.
\end{align}

A subscript is used to identify spaces of locally integrable functions or vector fields,
e.g. $U\in L^2_{\text{loc}}(\Omega)$ if and only if $\phi U$ is square-integrable for all
$\phi\in C^{\infty}_0(\mathbb{R}^3)$. Dual spaces, e.g. $H^1_0(\Omega_s)'=H^{-1}(\Omega_s)$, and dual operators, e.g. $(\gamma^-)'$ are written with primes. We use an asterisk to indicate spaces of functions
with zero mean, e.g. $H^1_*(\Omega)$, and let $\mathbf{mean}:H^1(\Omega_s)\rightarrow \mathbb{R}$ be the continuous operator defined by 
\begin{equation*}
    \mathbf{mean}(P):=\int_{\Omega_s}P(\mathbf{x})\dif\mathbf{x}.
\end{equation*}
Since its range is finite dimensional, $\mathbf{mean}$ is a compact operator \cite[Thm. 2.18]{MR1723850}. The operator $Q_*:H^1(\Omega_s)\rightarrow H^1_*(\Omega_s)$ defined by $Q_*=\id-\mathbf{mean}$ is a projection onto mean zero functions.

\subsection{Trace spaces}
Development of trace-related theory for Lipschitz domains and detailed definitions for the surface differential operators $\nabla_{\Gamma}$, $\text{curl}_{\Gamma}$, $\mathbf{curl}_{\Gamma}$ and $\text{div}_{\Gamma}$ can be found in \cite{buffa2001traces_a}, \cite{buffa2001traces_b} and \cite{buffa2002traces}.
In this section, we define the product trace spaces required for a variational
treatment of the Hodge--Laplace/Helmholtz operator. The traces are adapted to the system of equations at hand by accounting for the varying coefficients of \eqref{eq: strong from coupled sys a}.
\label{sec: classical traces}

Based on the continuous and surjective extensions
\begin{align*}
	\gamma &: H^1\bra{\Omega}\rightarrow H^{1/2}\bra{\Gamma}, &&\text{\cite[Thm. 4.2.1]{Hsiao2008}}\\
	\gamma_n &: \mathbf{H}(\text{div}, \Omega)\rightarrow H^{-1/2}\bra{\Gamma},&&\text{\cite[Thm. 2.5, Cor. 2.8]{girault2012finite}}\\
	\gamma_{\tau}&:\mathbf{H}\bra{\mathbf{curl},\Omega} \rightarrow \mathbf{H}^{-1/2}(\text{div}_\Gamma,\Gamma),&&\text{\cite[Thm. 4.1]{buffa2002traces}}\\
	\gamma_{t}&:\mathbf{H}\bra{\mathbf{curl},\Omega} \rightarrow \mathbf{H}^{-1/2}(\text{curl}_\Gamma,\Gamma), &&\text{\cite[Thm. 4.1]{buffa2002traces}}
\end{align*}
the traces previously introduced can also be extended by continuity to the relevant Sobolev spaces. We denote the duality pairing between $H^{1/2}(\Gamma)$ and $H^{-1/2}(\Gamma)$ by $\langle\cdot,\cdot\rangle_{\Gamma}$, but use $\langle\cdot,\cdot\rangle_{\tau}$ for the duality pairing between the trace spaces $\mathbf{H}^{-1/2}(\text{curl}_\Gamma,\Gamma)$ and $\mathbf{H}^{-1/2}(\text{div}_\Gamma,\Gamma)$ \cite[Lem. 5.6]{buffa2002traces}.

The duality pairings enter Green's formulas ($+$ for $\Omega=\Omega_s$)
 \begin{subequations}
\begin{align}
     \langle\gamma \bra{P} \gamma_n\bra{\mathbf{W}}\rangle_{\Gamma}&=\pm\int_{\Omega} \text{div}(\mathbf{W})\, P + \mathbf{W}\cdot\nabla P \, \dif \mathbf{x},\label{IBP div}\\
    \langle\gamma_{t}\bra{\mathbf{V}}, \gamma_{\tau}\bra{\mathbf{U}}\rangle_{\tau}&=\pm\int_{\Omega}\mathbf{U}\cdot\mathbf{curl\,}(\mathbf{V}) -\mathbf{curl\,}(\mathbf{U})\cdot\mathbf{V}\dif \mathbf{x}, \label{IBP curl ext}\\
    \langle \gamma_t(\mathbf{V}),\gamma_R(\mathbf{E}) \rangle_{\tau}&=\pm\int_{\Omega}\mathbf{curl}\,\mathbf{curl}\,\mathbf{E}\cdot\mathbf{V} - \mathbf{curl}\,\mathbf{E}\cdot \mathbf{curl}\,\mathbf{V}\dif\mathbf{x},\label{curl curl integral identity}
\end{align}
\end{subequations}
which hold for all $P\in H^1(\Omega)$, $\mathbf{W}\in\mathbf{H}(\text{div},\Omega)$, $\mathbf{U},\mathbf{V}\in\mathbf{H}(\mathbf{curl},\Omega)$ and $\mathbf{E}\in\mathbf{H}(\mathbf{curl}^2,\Omega)$.

\label{subsec: compound trace spaces}
As explained in \cite[Sec. 3]{claeys2017first}, a theory of differential equations for the
Hodge--Helmholtz/Laplace problem in three dimensions entails partitioning our collection
of traces into two \emph{dual} pairs. Accordingly, we now introduce the continuous and surjective mappings
  \begin{align*}
    &\mathcal{T}^{-}_{D,\epsilon}:\mathbf{H}_{\text{loc}}(\mathbf{curl},\Omega_s)\cap\mathbf{H}_{\text{loc}}(\epsilon;\nabla\text{div},\Omega_s)\rightarrow\mathcal{H}_D(\Gamma),\\
    &\mathcal{T}^{+}_{D}:\mathbf{H}_{\text{loc}}(\mathbf{curl},\Omega')\cap\mathbf{H}_{\text{loc}}(\nabla\text{div},\Omega')\rightarrow\mathcal{H}_D(\Gamma),\\
    &\mathcal{T}^{-}_{N,\mu}:\mathbf{H}_{\text{loc}}(\mu^{-1};\mathbf{curl}^2,\Omega_s)\cap\mathbf{H}_{\text{loc}}(\epsilon;\text{div},\Omega_s) \rightarrow\mathcal{H}_N(\Gamma),\\
    &\mathcal{T}^{+}_{N}:\mathbf{H}_{\text{loc}}(\mathbf{curl}^2,\Omega')\cap\mathbf{H}_{\text{loc}}(\text{div},\Omega') \rightarrow\mathcal{H}_N(\Gamma),
  \end{align*}
defined by 
\begin{align*}
  \mathcal{T}^{-}_{D,\epsilon}\bra{\mathbf{U}}&:= \colvec{2}{\gamma^-_t(\mathbf{U})}{\gamma_{D,\epsilon}^-(\mathbf{U})},
  &
  \mathcal{T}^{-}_{N,\mu}\bra{\mathbf{U}}&:= \colvec{2}{\gamma_{R,\mu}(\mathbf{U})}{\gamma^-_{n,\epsilon}(\mathbf{U})},\\
  \mathcal{T}^+_D(\mathbf{U})&:=\colvec{2}{\gamma_t^+(\mathbf{U})}{\gamma^+_{D,\eta}(\mathbf{U})}, &
  \mathcal{T}^+_N(\mathbf{U})&:=\colvec{2}{\gamma_R(\mathbf{U})}{\gamma_n(\mathbf{U})},
\end{align*}
where
  \begin{align*}
    &\mathcal{H}_D:=\mathbf{H}^{-1/2}(\text{curl}_{\Gamma},\Gamma)\times H^{1/2}(\Gamma),\\
    &\mathcal{H}_N:=\mathbf{H}^{-1/2}(\text{div}_{\Gamma},\Gamma)\times H^{-1/2}(\Gamma).
  \end{align*}
They admit continuous right-inverses, i.e. lifting maps from the trace spaces into $\mathbf{X}(\Delta,\Omega)$ \cite[Lem. 3.2]{claeys2017first}.

In literature the pair of traces involved in $\mathcal{T}_N$ is labelled as \emph{magnetic}, while the pair in $\mathcal{T}_D$ is referred to as \emph{electric}---simply because one recovers the magnetic field by taking the curl of the potential $\mathbf{U}$. However, our choice of subscripts is motivated by the analogy between this pair of product traces and the classical Dirichlet and Neumann boundary conditions for second-order elliptic BVPs. 

The trace spaces $\mathcal{H}_D$ and $\mathcal{H}_N$ are put in duality using the sum of the inherited component-wise duality parings. That is, for $\vec{\mathbf{p}}=(\mathbf{p},q)\in\mathcal{H}_N$ and $\vec{\bm{\eta}}=(\bm{\eta},\zeta)\in\mathcal{H}_D$, we define 
\begin{equation*}
  \langle \vec{\mathbf{p}}, \vec{\bm{\eta}}\rangle := \langle\mathbf{p},\bm{\eta}\rangle_{\tau} + \langle q, \zeta\rangle_{\Gamma}.
\end{equation*}

We indicate with curly brackets the average
\begin{equation*}
\{\gamma_{\bullet}\}:= \frac{1}{2}\bra{\gamma_{\bullet}^+ + \gamma_{\bullet}^-}
\end{equation*}
 of a trace and with square brackets its jump
 \begin{equation*}
 [\gamma_{\bullet}]:=\gamma_{\bullet}^- - \gamma_{\bullet}^+
 \end{equation*}
 over the interface $\Gamma$, $\bullet=R$, $D$, $t$, $\tau$, or $n$. Corresponding notation is used for the product traces.
\begin{warning}
	Notice the sign in the jump $[\gamma]=\gamma^- - \gamma^+$, which is often taken to be the opposite in literature!
\end{warning}

\subsection{Boundary potentials}\label{sec: Boundary potentials}
By exploiting the radiating fundamental solution 
\begin{equation*}
  G_{\nu}(\mathbf{x}):=\exp\left(i\nu\abs{\mathbf{x}}\right)/4\pi\abs{\mathbf{x}}
\end{equation*}
for the scalar Helmholtz operator $\Delta-\nu^2\id$, it is shown in \cite[Sec. 4.2]{claeys2017first} that a distributional solution $\mathbf{U}\in\mathbf{L}^2(\mathbb{R}^3)$ such that $\mathbf{U}\vert_{\Omega_s}\in\mathbf{X}(\Delta,\Omega_s)$ and $\mathbf{U}\vert_{\Omega'}\in\mathbf{X}_{\text{loc}}(\Delta,\Omega')$ of the homogeneous (scaled) Hodge--Helmholtz/Laplace equation \eqref{eq: strong from coupled sys b} with constant coefficients $\eta>0$, $\kappa\geq 0$, stated in the whole of $\mathbb{R}^3$ with radiation conditions at infinity as considered in Section \ref{sec: Introduction}, affords a representation formula
\begin{equation}
  \mathbf{U} = \mathcal{SL}_{\kappa}\cdot [\mathcal{T}_{N}(\mathbf{U})] + \mathcal{DL}_{\kappa}\cdot [\mathcal{T}_{D}(\mathbf{U})]\qquad\qquad\text{in }\mathbb{R}^3\backslash\Gamma.\label{representation formula}
\end{equation}

Letting $\tilde{\kappa}=\kappa/\sqrt{n}$, the Hodge-Helmholtz single layer potential is explicitly given by

\begin{equation}\label{Hodge-Helmholtz single layer potential}
  \mathcal{SL}_{\kappa}\bra{\colvec{2}{\mathbf{p}}{q}}=-\bm{\Psi}_{\kappa}(\mathbf{p}) - \nabla\tilde{\psi}_k\bra{\text{div}_{\Gamma}(\mathbf{p})} + \nabla\psi_{\tilde{\kappa}}(q),
\end{equation}
where the Helmholtz scalar single-layer, vector single-layer and the regular potentials are written individually for  $\mathbf{p}\in\mathbf{H}^{-1/2}(\text{div}_{\Gamma},\Gamma)$ and $q\in H^{-1/2}(\Gamma)$ as
\begin{subequations}
  \begin{align}
    \psi_{\nu}(q)(\mathbf{x})&:=\int_{\Gamma}q(\mathbf{y})G_{\nu}(\mathbf{x}-\mathbf{y})\dif\sigma(\mathbf{y}), &\mathbf{x}\in\mathbb{R}^3\backslash\Gamma, \label{scalar single layer}\\
    \bm{\Psi}_{\nu}(\mathbf{p})(\mathbf{x})&:=\int_{\gamma}\mathbf{p}(\mathbf{y})G_{\nu}(\mathbf{x}-\mathbf{y})\dif\sigma(\mathbf{y}), &\mathbf{x}\in\mathbb{R}^3\backslash\Gamma,\label{vector single layer}\\
    \tilde{\psi}_{\kappa}(q)(\mathbf{x})&:=\int_{\Gamma}q(\mathbf{y})\frac{G_{\kappa}-G_{\tilde{\kappa}}}{\kappa^2}(\mathbf{x}-\mathbf{y})\dif\sigma(\mathbf{y}), &\mathbf{x}\in\mathbb{R}^3\backslash\Gamma,\label{regular potential}
  \end{align}
\end{subequations}
respectively. The expression \eqref{Hodge-Helmholtz single layer potential} is derived with \eqref{scalar single layer}-\eqref{regular potential} understood as duality pairings. However, if the essential supremum of $\mathbf{p}$, $q$ and $\text{div}_{\Gamma}(\mathbf{p})$ is bounded, then they can safely be computed as improper integrals \cite[Rmk. 4.2]{claeys2017first}. These classical potentials satisfy
\begin{subequations}
  \begin{align}
    -\,\text{div}\,\nabla \psi_{\tilde{\kappa}}(q)&= \tilde{\kappa}^2\psi_{\tilde{\kappa}}(q),\label{scalar helmholtz for scalar single layer}\\
    -\,\Delta \bm{\Psi}_{\kappa}(\mathbf{p}) &= \kappa^2\bm{\Psi}_{\kappa}(\mathbf{p})\label{vector helmholtz for vector single layer},\\
    -\text{div}\,\nabla\tilde{\psi}_{\kappa}(q) &= \psi_{\kappa}(q) + \frac{1}{\eta}\psi_{\tilde{\kappa}}(q),
  \end{align}
\end{subequations}
and the identity \cite[Lem. 2.3]{maccamy1984solution}
\begin{align}\label{scalar and vector single layer potential identity}
  \text{div}\,\bm{\Psi}_{\nu}(\mathbf{p}) &= \psi_{\nu}\left(\text{div}_{\Gamma}\mathbf{p}\right) & &\forall\,\mathbf{p}\in \mathbf{H}^{-1/2}(\text{div}_\Gamma,\Gamma).
\end{align}
The mapping properties of $\psi_{\nu}$, $\bm{\Psi}_{\nu}$,  $\nabla\psi_{\tilde{\kappa}}$ and $\nabla\tilde{\psi}_{\kappa}$ are detailed in \cite[Sec. 5]{claeys2017first}.

Ultimately, we will resort to a Fredholm alternative argument to prove well-posedness of the coupled system. It is therefore evident that the compactness properties of the boundary integral operators introduced in the next Lemma will be extensively used both explicitly and implicitly---notably through exploiting the results found in \cite[Sec. 6]{claeys2017first}.

From \cite[Lem. 3.9.8]{sauter2010boundary} and \cite[Lem. 7]{buffa2003galerkin}, we know that for any $\nu\geq 0$, the following operators are compact:
\begin{subequations}
    \begin{align}
      \gamma^{\pm}\bra{\psi_{\nu}-\psi_{0}}:&\,H^{-1/2}(\Gamma)\rightarrow H^{1/2}(\Gamma),\label{eq: compact nu 0 phi}\\
      \gamma^{\pm}_n\bra{\nabla\psi_{\nu}-\nabla\psi_{0}}:&\,H^{-1/2}(\Gamma)\rightarrow H^{-1/2}(\Gamma),\\
      \gamma^{\pm}_{t}\bra{\bm{\Psi}_{\nu}-\bm{\Psi}_{0}}:&\,\mathbf{H}^{-1/2}(\text{div}_{\Gamma},\Gamma)\rightarrow \mathbf{H}^{-1/2}(\text{curl}_{\Gamma},\Gamma),\\
      \gamma^{\pm}_n\nabla\tilde{\psi}_{\nu}:&\,H^{-1/2}(\Gamma)\rightarrow H^{-1/2}(\Gamma).\label{eq: compact nu 0 phi tilde}\
    \end{align}
    \end{subequations}
Compactness of the second boundary integral operator listed immediately entails compactness of 
$$\nu^2\gamma^{\pm}_n\nabla\tilde{\psi}_{\nu}=\gamma^{\pm}_n\bra{\nabla\psi_{\nu}-\nabla\psi_{\tilde{\nu}}}=\gamma^{\pm}_n\bra{\nabla\psi_{\nu}-\nabla\psi_{0}} - \bra{\gamma^{\pm}_n\bra{\nabla\psi_{\tilde{\nu}}-\nabla\psi_{0}}}$$
by linearity. While it seems that blow-up occurs in $\tilde{\psi}_{\nu}$ as $\nu\rightarrow 0$, $\nabla\tilde{\psi}_{\nu}$ happens to be an entire function of $\nu$ that vanishes at $\nu = 0$ \cite[Sec. 4.1]{claeys2017first}.

The Hodge--Helmholtz double layer potential is given for boundary data $\bm{\eta}\in\mathbf{H}^{-1/2}(\text{curl}_{\Gamma},\Gamma)$ and $\xi\in H^{1/2}(\Gamma)$ by
\begin{equation}\label{double layer potential}
  \mathcal{DL}_{\kappa}\bra{\colvec{2}{\bm{\eta}}{\xi}} := \mathbf{curl}\,\bm{\Psi}_{\kappa}(\bm{\eta}\times\mathbf{n}) + \Upsilon_{\kappa}(\xi).
\end{equation}
We recognize in \eqref{double layer potential} the (electric) Maxwell double layer potential (c.f. \cite[Sec. 4]{hiptmair2003coupling}, \cite[Eq. 28]{buffa2003galerkin}) and the normal vector single-layer potential
\begin{align*}
  \Upsilon_{\kappa}(\xi)&:=\int_{\Gamma}\xi(\mathbf{y})\mathbf{G}_{\kappa}(\mathbf{x}-\mathbf{y})\mathbf{n}(\mathbf{y})\dif\sigma(\mathbf{y}), & \mathbf{x}\in\mathbb{R}^3\backslash\Gamma,
\end{align*}
in which appears the matrix-valued fundamental solution 
\begin{equation*}
  \mathbf{G}_{\kappa} := G_{\kappa}\id +{\kappa^{-2}}\nabla^2\bra{G_{\kappa}-G_{\tilde{\kappa}}}
\end{equation*}
satisfying $-\Delta_{\eta}\mathbf{G}_{\kappa} - \kappa^2\mathbf{G}_{\kappa}=\delta_0\,\id$ exploited in \cite{claeys2017first} and detailed in \cite[App. A]{hazard1996solution}. This surface potential satisfies
\begin{equation}
  \label{eq: Hodge-Helmholtz Upsilon}
  -\Delta_{\eta}\Upsilon_{\kappa}(\xi)=\kappa^2\Upsilon_{\kappa}(\xi)
\end{equation}
and the identity \cite[Sec.5.4]{claeys2017first}
\begin{math}
  \mathbf{curl}\,\Upsilon_{\kappa}(\xi)=\mathbf{curl}\bm{\Psi}_{\kappa}(\xi\mathbf{n})
\end{math}.

The mapping properties of the potentials
$\mathbf{curl}\bm{\Psi}_{\kappa}(\cdot\times\mathbf{n})$ and $ \Upsilon_{\kappa}$ are
detailed in \cite[Sec. 5]{claeys2017first}.

\subsection{Integral operators}\label{sec: Integral operators}
In this section, we extend the analysis performed in \cite{buffa2003galerkin,hiptmair2003coupling} for the classical electric wave equation to the boundary integral operators arising from Hodge--Helmholtz and Hodge--Laplace problems.

The well-known Cald\'eron identities are obtained from \eqref{representation formula} upon taking the classical compounded traces on both sides and utilizing the jump relations
\begin{subequations}
  \begin{align}
    [\mathcal{T}_D]\cdot\mathcal{DL}_{\kappa}(\vec{\bm{\eta}}) &= \vec{\bm{\eta}}, & [\mathcal{T}_N]\cdot\mathcal{DL}_{\kappa}(\vec{\bm{\eta}}) &= 0,\quad\quad \vec{\bm{\eta}}\in\mathcal{H}_D, \label{double layer jump}\\ 
    [\mathcal{T}_D]\cdot\mathcal{SL}_{\kappa}(\vec{\mathbf{p}}) &= 0, & [\mathcal{T}_N]\cdot\mathcal{SL}_{\kappa}(\vec{\mathbf{p}}) &= \vec{\mathbf{p}},\quad\quad \vec{\mathbf{p}}\in\mathcal{H}_N, \label{single layer jump}
  \end{align}
\end{subequations}
given in \cite[Thm. 5.1]{claeys2017first}. The operator forms of the interior and exterior Cald\'eron projectors defined on $\mathcal{H}_D\times\mathcal{H}_N$, which we denote $\mathbb{P}^-_{\kappa}$ and $\mathbb{P}^+_{\kappa}$ respectively, enter the Cald\'eron identites:
\begin{subequations}
  \begin{gather}
    \underbrace{\begin{pmatrix}
        \{\mathcal{T}_D\}\cdot\mathcal{DL}_k+\frac{1}{2}\id & \{\mathcal{T}_D\}\cdot\mathcal{SL}_k\\
        \{\mathcal{T}_N\}\cdot\mathcal{DL}_k & \{\mathcal{T}_N\}\cdot\mathcal{SL}_k+\frac{1}{2}\id
      \end{pmatrix}}_{=:\mathbb{P}^-_{\kappa}}
    \begin{pmatrix}
      \mathcal{T}_D^-\mathbf{U}\\
      \mathcal{T}_N^-\mathbf{U}
    \end{pmatrix}
    = \begin{pmatrix}
      \mathcal{T}_D^-\mathbf{U}\\
      \mathcal{T}_N^-\mathbf{U}
    \end{pmatrix},\label{interior calderon}\\
    \underbrace{\begin{pmatrix}
        -\{\mathcal{T}_D\}\cdot\mathcal{DL}_k+\frac{1}{2}\id & -\{\mathcal{T}_D\}\cdot\mathcal{SL}_k\\
        -\{\mathcal{T}_N\}\cdot\mathcal{DL}_k & -\{\mathcal{T}_N\}\cdot\mathcal{SL}_k+\frac{1}{2}\id
      \end{pmatrix}}_{=:\mathbb{P}^+_{\kappa}}
    \begin{pmatrix}
      \mathcal{T}_D^+\mathbf{U}^{\text{ext}}\\
      \mathcal{T}_N^+\mathbf{U}^{\text{ext}}
    \end{pmatrix}
    = \begin{pmatrix}
      \mathcal{T}_D^+\mathbf{U}^{\text{ext}}\\
      \mathcal{T}_N^+\mathbf{U}^{\text{ext}}
    \end{pmatrix},\label{exterior calderon}
  \end{gather}
\end{subequations}

Note that $\mathbb{P}^-_{\kappa} + \mathbb{P}^+_{\kappa}=\id$ and that the range of
$\mathbb{P}^+_{\kappa}$ coincides with the kernel of $\mathbb{P}^-_{\kappa}$ and
vice-versa \cite[Sec. 5]{buffa2003galerkin}. As a consequence of the jump relations
\eqref{double layer jump}-\eqref{single layer jump}, the representation formula
\eqref{representation formula} and the existence of trace liftings, the pair of
``magnetic" and ``electric" traces
$\bra{\vec{\bm{\eta}}\,\,\,\vec{\mathbf{p}}}^{\top}\in\mathcal{H}_D\times\mathcal{H}_N$ is
valid interior or exterior Cauchy data, if and only if it lies in the kernel of
$\mathbb{P}^+_{\kappa}$ or $\mathbb{P}^-_{\kappa}$ respectively
(c.f.\cite[Lem. 6.18]{steinbach2007numerical}, \cite[Thm. 8]{buffa2003galerkin} and
\cite[Prop. 5.2]{claeys2017first}).

Inspecting equations \eqref{interior calderon}-\eqref{exterior calderon} reveals that the
Cald\'eron projectors share a common structure. They can be written as
\begin{align*}
    \mathbb{P}_{\kappa}^- = \frac{1}{2}\id + \mathbb{A}_\kappa
    &&\text{and} 
    &&\mathbb{P}_{\kappa}^+ = \frac{1}{2}\id - \mathbb{A}_\kappa,
\end{align*} and
where the Cald\'eron operator
$\mathbb{A}_\kappa:\mathcal{H_D\times\mathcal{H}_N\to\mathcal{H}_D\times\mathcal{H}_N}$ is given by
\begin{gather}
  \label{eq: Ak blocks}
  \mathbb{A}_\kappa :=\begin{pmatrix}
    \mathbb{A}_{\kappa}^{DD} & \mathbb{A}_{\kappa}^{ND}\\
    \mathbb{A}_{\kappa}^{DN} & \mathbb{A}_{\kappa}^{NN}
  \end{pmatrix} :=
  \begin{pmatrix}
    \{\mathcal{T}_D\}\cdot\mathcal{DL}_\kappa & \{\mathcal{T}_D\}\cdot\mathcal{SL}_\kappa\\
    \{\mathcal{T}_N\}\cdot\mathcal{DL}_\kappa & \{\mathcal{T}_N\}\cdot\mathcal{SL}_\kappa
  \end{pmatrix}.
\end{gather}

An analog of the operator matrix $\mathbb{A}_k$ was found convenient in the study of the
boundary integral equations of electromagnetic scattering problems
\cite[Sec. 6]{buffa2003galerkin}. It is known from \cite{claeys2017first} that the
off-diagonal blocks $\mathbb{A}_{\kappa}^{DN}$ and $\mathbb{A}_{\kappa}^{ND}$ of $\mathbb{A}_k$ independently satisfy generalized G{\aa}rding
inequalities making them of Fredholm type with index 0. Injectivity holds when $\kappa^2$
lies outside a discrete set of ``forbidden resonant frequencies'' accumulating at infinity
\cite[Sec. 3]{claeys2017first}. More explanations will be given in Section 3. In the static case $\kappa=0$, the dimensions of
$\ker\left(\{\mathcal{T}_N\}\cdot\mathcal{SL}_0\right)$ and
$\ker\left(\{\mathcal{T}_D\}\cdot\mathcal{DL}_0\right)$ agree with the
zeroth and first Betti number of $\Gamma$, respectively \cite[Sec. 7]{claeys2017first}.

In the case of the classical electric wave equation, the boundary integral operators
involved in the Cald\'eron projectors enjoy a hidden symmetry: there exists a compact linear operator
  $\mathbf{C}_k:\mathbf{H}^{-1/2}(\text{div}_\Gamma,\Gamma)\rightarrow
  \mathbf{H}^{-1/2}(\text{div}_\Gamma,\Gamma)$ such that
  \begin{equation}\label{eq: adjoint maxwell potential}
    \ip{\{\gamma_R\}\bm{\Psi}_k(\mathbf{p})}{\bm{\eta}}_{\tau} = \ip{\mathbf{p}}{\{\gamma_t\}\bm{\Psi}_{\kappa}\mathbf{curl}(\bm{\eta}\times\mathbf{n})}_{\tau} +\ip{\mathbf{C}_k\mathbf{p}}{\bm{\eta}}_{\tau}
  \end{equation}
  for all $\mathbf{p}\in\mathbf{H}^{-1/2}(\text{div}_\Gamma,\Gamma)$ and $\bm{\eta}\in\mathbf{H}^{-1/2}(\mathbf{curl}_\Gamma,\Gamma)$, cf. \cite[Lem. 5.4]{hiptmair2003coupling} and \cite[Lem. 6]{buffa2003galerkin}.

We will extend this result to the integral operators
defined for the scaled Hodge--Helmholtz/Laplace equation to better characterize the
structure of \eqref{eq: Ak blocks}. The symmetry we are about to reveal in the diagonal
blocks $\mathbb{A}_{\kappa}^{NN}$ and $\mathbb{A}_{\kappa}^{DD}$ of the Cald\'eron projectors will be crucial in the derivation of the main
T-coercivity estimate of this work. It will be exploited for complete \emph{cancellation}, \emph{up to
compact terms,} of the operators lying on the \emph{off-diagonal} of the block operator matrix
associated to the coupled variational system introduced in Section \ref{sec: coupled
  problem}. The following lemmas are required.
\begin{lemma}\label{lem: grad phi term is anti hermitian}
  There is a compact linear operator $C_k:H^{-1/2}(\Gamma)\rightarrow H^{-1/2}(\Gamma)$ such that 
  \begin{equation*}
    \langle\{\gamma_n\}\nabla\psi_{\tilde{\kappa}}(q), \xi\rangle_{\Gamma} =-\ip{q}{\{\eta\,\gamma_D\}\Upsilon_{\kappa}(\xi)}_{\Gamma} + \ip{C_kq}{\xi}_{\Gamma},
  \end{equation*}
  for all $q\in H^{-1/2}(\Gamma)$, $\xi\in H^{1/2}(\Gamma)$.
\end{lemma}
\begin{proof}
  This proof utilizes a strategy found in \cite[Lem. 5.4]{hiptmair2003coupling} and
  \cite[Thm. 3.9]{Buffa2002Boundary}. Let $\rho > 0 $ be such that $B_{\rho}$ is an open
  ball containing $\overline{\Omega}_s$. We will indicate with a hat
  (e.g. $\widehat{\gamma}$) the traces taken over the boundary $\partial B_{\rho}$ of that
  ball and use Green's formula to compare the following terms.
  
  On the one hand, using the scalar Helmholtz equation \eqref{scalar helmholtz for scalar single layer} and recalling that $\tilde{\kappa}=\kappa/\sqrt{\eta}$, we have
  \begin{align}
     &\langle\eta\,\hlo{ \gamma_D^-}\nabla\psi_{\tilde{\kappa}}(q), \hlo{\gamma_n^-}\Upsilon_{\kappa}(\xi)\rangle_{\Gamma}\nonumber\\
    &\qquad\qquad= \int_{\Omega_s}\eta\,\text{div}\left(\nabla\psi_{\tilde{\kappa}}(q)\right)\text{div}\Upsilon_k(\xi) +\eta\nabla\text{div}\left(\nabla\psi_{\tilde{\kappa}}(q)\right)\cdot\Upsilon_{\kappa}(\xi)\dif\mathbf{x}\nonumber\\
    &\qquad\qquad= \hly{-\int_{\Omega_s}\kappa^2\psi_{\tilde{\kappa}}(q)\text{div}\Upsilon_k(\xi)\dif\mathbf{x} -\int_{\Omega_s} \kappa^2\nabla\psi_{\tilde{\kappa}}(q)\cdot\Upsilon_{\kappa}(\xi)\dif\mathbf{x},}\label{eq: interior eq 1}
  \end{align}
  and similarly,
  \begin{align*}
    \langle\eta\,\gamma_D^+\nabla\psi_{\tilde{\kappa}}(q), \gamma_n^+\Upsilon_{\kappa}(\xi)\rangle_{\Gamma} &= \int_{\Omega'\cap B_{\rho}}\kappa^2\psi_{\tilde{\kappa}}(q)\,\text{div}\Upsilon_k(\xi) + \nabla\psi_{\tilde{\kappa}}(q)\cdot\Upsilon_{\kappa}(\xi)\dif\mathbf{x} \\
    &\qquad\qquad\hlb{+\langle\eta\,\widehat{\gamma}_D^+\nabla\psi_{\kappa}(q), \widehat{\gamma}_n^+\Upsilon_{\kappa}(\xi)\rangle_{\partial B_{\rho}}}.
  \end{align*}
  On the other hand, using \eqref{scalar helmholtz for scalar single layer} together with the scaled Hodge--Helmholtz equation \eqref{eq: Hodge-Helmholtz Upsilon}, we also have
  \begin{align}
    \langle\hlo{\gamma_n^-}&\nabla\psi_{\tilde{\kappa}}(q),
    \eta\,\hlo{\gamma_D^-}\Upsilon_{\kappa}(\xi)\rangle_{\Gamma}
    \nonumber\\
    &=  \int_{\Omega_s} \eta\,\text{div}\left(\nabla\psi_{\tilde{\kappa}}(q)\right)\text{div}\Upsilon_{\kappa}(\xi)\dif\mathbf{x}
    +
    \int_{\Omega_s}\eta\,\nabla\psi_{\tilde{\kappa}}(q)\cdot\nabla\text{div}\Upsilon_{\kappa}(\xi)\dif{\mathbf{x}}
    \nonumber\\
    &=
      \hly{-\int_{\Omega_s}\kappa^2\psi_{\tilde{\kappa}}(q)\,\text{div}\Upsilon_{\kappa}(\xi)\dif\mathbf{x}}
      +
      \hlm{\int_{\Omega_s}\nabla\psi_{\tilde{\kappa}}(q)\cdot\mathbf{curl}\,\mathbf{curl}\,\Upsilon_{\kappa}(\xi)\dif{\mathbf{x}}}
      \nonumber\\
      &\qquad\qquad \hly{-\int_{\Omega_s}\kappa^2\nabla\psi_{\tilde{\kappa}}(q)\cdot\Upsilon_{\kappa}(\xi)\dif{\mathbf{x}}}.\label{eq: interior eq 2}
    \end{align}
  Equations \eqref{eq: interior eq 1} and \eqref{eq: interior eq 2} together yield
  \begin{multline*}
    \langle\hlo{\gamma_n^-}\nabla\psi_{\tilde{\kappa}}(q), \eta\,\hlo{\gamma_D^-}\Upsilon_{\kappa}(\xi)\rangle_{\Gamma} = \langle\eta\,\hlo{\gamma_D^-}\nabla\psi_{\kappa}(q), \hlo{\gamma_n^-}\Upsilon_{\kappa}(\xi)\rangle_{\Gamma}\\ +\hlm{\int_{\Omega_s}\nabla\psi_{\tilde{\kappa}}(q)\cdot\mathbf{curl}\,\mathbf{curl}\,\Upsilon_{\kappa}(\xi)\dif{\mathbf{x}}}.
  \end{multline*}
  Similarly, the terms involving the exterior traces satisfy
  \begin{align*}
    \langle\gamma_n^+\nabla\psi_{\tilde{\kappa}}(q), \eta\,\gamma_D^+\Upsilon_{\kappa}(\xi)\rangle_{\Gamma} &=\langle\eta\,\gamma_D^+\nabla\psi_{\kappa}(q), \gamma_n^+\Upsilon_{\kappa}(\xi)\rangle_{\Gamma}\\
    &\quad\qquad- \hlb{\langle\eta\,\widehat{\gamma}_D^+\nabla\psi_{\kappa}(q), \widehat{\gamma}_n^+\Upsilon_{\kappa}(\xi)\rangle_{\partial B_{\rho}}}\\
    &\quad\qquad-\hlm{\int_{\Omega'\cap B_{\rho}}\nabla\psi_{\tilde{\kappa}}(q)\cdot\mathbf{curl}\,\mathbf{curl}\,\Upsilon_{\kappa}(\xi)\dif{\mathbf{x}}}\\
    &\quad\qquad+ \hlb{\langle\widehat{\gamma}_n^+\nabla\psi_{\tilde{\kappa}}(q), \eta\,\widehat{\gamma}_D^+\Upsilon_{\kappa}(\xi)\rangle_{\partial B_{\rho}}.}
  \end{align*}
  From the first row of the jump properties \cite[Sec. 5]{claeys2017first}
  \begin{subequations}
    \begin{align}
      &[\gamma_D]\nabla\psi_{\tilde{\kappa}}(q) = 0, & &[\gamma_n]\Upsilon_{\kappa}(\xi) = 0,\\
      &[\gamma_D]\Upsilon_{\kappa}(\xi) = \xi/\eta, & &[\gamma_n]\nabla\psi_{\tilde{\kappa}}(q)=q \label{non-vanishing jumps psi upsilon},
    \end{align}
  \end{subequations}
  we obtain, by gathering the above results, integrating by parts again and using the fact that $\mathbf{curl}\circ\nabla\equiv 0$, 
  \begin{align}
    \label{switching is compact perturbation psi upsilon}
     \langle\gamma_n^-&\nabla\psi_{\tilde{\kappa}}(q),
    \eta\,\gamma_D^-\Upsilon_{\kappa}(\xi)\rangle_{\Gamma} \nonumber \\ & = 
    \langle\eta\,\gamma_D^+\nabla\psi_{\kappa}(q),
    \gamma_n^+\Upsilon_{\kappa}(\xi)\rangle_{\Gamma}\nonumber
    +
    \int_{\Omega_s}\kappa^2\nabla\psi_{\tilde{\kappa}}(q)\cdot\bm{\Psi}_{\kappa}(\xi\mathbf{n})\dif\mathbf{x}
    \\
    &=  \langle\gamma_n^+\nabla\psi_{\tilde{\kappa}}(q),
    \eta\,\gamma_D^+\Upsilon_{\kappa}(\xi)\rangle_{\Gamma} +
    \hlm{\int_{B_{\rho}}\nabla\psi_{\tilde{\kappa}}(q)\cdot\mathbf{curl}\,\mathbf{curl}\,\Upsilon_{\kappa}(\xi)\dif{\mathbf{x}}}\nonumber
    \\
    &\qquad\hlb{+ \langle\eta\,\widehat{\gamma}_D^+\nabla\psi_{\kappa}(q),
      \widehat{\gamma}_n^+\Upsilon_{\kappa}(\xi)\rangle_{\partial B_{\rho}}
      -\langle\widehat{\gamma}_n^+\nabla\psi_{\tilde{\kappa}}(q),
      \eta\,\widehat{\gamma}_D^+\Upsilon_{\kappa}(\xi)\rangle_{\partial B_{\rho}}.}
    \nonumber\\
    &=  \langle\gamma_n^+\nabla\psi_{\tilde{\kappa}}(q),
    \eta\,\gamma_D^+\Upsilon_{\kappa}(\xi)\rangle_{\Gamma} +
    \hlm{\langle\gamma_t\nabla\psi_{\tilde{\kappa}}(q),
      \gamma_R\Upsilon_{\kappa}(\xi)\rangle_{\partial B_{\rho}}}\nonumber
    \\
    &\qquad\hlb{+ \langle\eta\,\widehat{\gamma}_D^+\nabla\psi_{\kappa}(q), \widehat{\gamma}_n^+\Upsilon_{\kappa}(\xi)\rangle_{\partial B_{\rho}} -\langle\widehat{\gamma}_n^+\nabla\psi_{\tilde{\kappa}}(q), \eta\,\widehat{\gamma}_D^+\Upsilon_{\kappa}(\xi)\rangle_{\partial B_{\rho}}.}
  \end{align}
  Fortunately, when restricted to domains away from $\Gamma$, the potentials are $C^{\infty}$-smoothing. Hence, their evaluation on $\partial B_{\rho}$, the highlighted terms in \eqref{switching is compact perturbation psi upsilon}, induce compact operators. This shows that for some compact operator $C_k:H^{-1/2}(\Gamma)\rightarrow H^{-1/2}(\Gamma)$,
  \begin{equation}\label{identity with compact upsilon psi}
    \langle\gamma_n^-\nabla\psi_{\tilde{\kappa}}(q), \eta\,\gamma_D^-\Upsilon_{\kappa}(\xi)\rangle_{\Gamma} = \langle\gamma_n^+\nabla\psi_{\tilde{\kappa}}(q), \eta\,\gamma_D^+\Upsilon_{\kappa}(\xi)\rangle_{\Gamma} + \ip{C_k q}{\xi}_{\Gamma}.
  \end{equation}
  
  The jump identities \eqref{non-vanishing jumps psi upsilon} for the potentials yield formulas of the form $\{\gamma_{\bullet}\}K = \gamma_{\bullet}^{\pm}K\pm (1/2)\id$, where $\bullet=n,\,D$ and $K=\nabla\psi_{\tilde{\kappa}},\,\Upsilon_{\kappa}$ accordingly. Substituting each one-sided trace involved in the two leftmost duality pairings of \eqref{identity with compact upsilon psi} for the integral operators using these equations completes the proof.
\end{proof}
\begin{lemma}\label{lem: symmetry of two last terms single layer}
  For all  $\mathbf{p}\in\mathbf{H}^{-1/2}(\emph{div}_\Gamma,\Gamma)$ and $\xi\in H^{1/2}(\Gamma)$, we have
  \begin{equation*}
    \ip{\mathbf{p}}{\gamma^{\pm}_t\Upsilon_{\kappa}(\xi)}_{\tau}  = \langle\gamma^{\pm}_n\bm{\Psi}_{\kappa}(\mathbf{p}),\xi \rangle_{\Gamma} + \langle\gamma^{\pm}_n\nabla\tilde{\psi}_{\kappa}(\text{\emph{div}}_{\Gamma}(\mathbf{p})), \xi\rangle_{\Gamma}.
  \end{equation*} 
\end{lemma}
\begin{proof}
  In the following calculations, the boundary integrals are to be understood as duality pairings. Since $\mathbf{p}\in\mathbf{L}^2_t(\Gamma)$ is a tangent vector field lying in the image of $\gamma_t$, the tangential trace operator can safely be dropped in expanding these integrals using the definitions of Section \ref{sec: Boundary potentials}. On the one hand, this leads to
  \begin{align*}
    \ip{\mathbf{p}}{\gamma^{\pm}_t\Upsilon_{\kappa}(\xi)}_{\tau}  &= \int_{\Gamma}\int_{\Gamma}\xi(\mathbf{y})\mathbf{p}(\mathbf{x})\cdot\bra{\mathbf{G}_{\kappa}(\mathbf{x}-\mathbf{y})\mathbf{n}(\mathbf{y})}\dif\sigma(\mathbf{y})\dif\sigma(\mathbf{x})\\
    &= \int_{\Gamma}\int_{\Gamma}\xi(\mathbf{y})G_{\kappa}(\mathbf{x}-\mathbf{y})\mathbf{p}(\mathbf{x})\cdot \mathbf{n}(\mathbf{y})\dif\sigma(\mathbf{y})\dif\sigma(\mathbf{x})\\
    &\qquad\hly{+\int_{\Gamma}\int_{\Gamma}\xi(\mathbf{y})\mathbf{p}(\mathbf{x})\cdot \bra{\nabla^2\tilde{G}_{\kappa}(\mathbf{x}-\mathbf{y})\mathbf{n}(\mathbf{y})}\dif\sigma(\mathbf{y})\dif\sigma(\mathbf{x}),}
  \end{align*}
  where $\tilde{G}_{\kappa}:=\bra{G_{\kappa}-G_{\tilde{\kappa}}}/\kappa^2$.
  
  On the other hand, the same observation implies that $\langle \mathbf{p},\nabla_\Gamma\gamma\mathbf{V})\rangle_{\tau}=\langle \mathbf{p},\gamma\nabla\mathbf{V})\rangle_{\tau}$ for any $\mathbf{V}\in\mathbf{H}_{\text{loc}}^1(\mathbb{R}^3)$, and thus that
  \begin{align*}
    \langle\gamma^{\pm}_n\nabla\tilde{\psi}_{\kappa}(\text{div}_{\Gamma}(\mathbf{p})), \xi\rangle_{\Gamma} &=\int_{\gamma}\int_{\gamma}\xi(\mathbf{y})\mathbf{n}(\mathbf{y})\cdot\nabla\tilde{G}_{\kappa}\bra{\mathbf{y}-\mathbf{x}}\text{div}_{\Gamma}\bra{\mathbf{p}(\mathbf{x})}\dif\sigma(\mathbf{y})\dif\sigma(\mathbf{x})\\
    &=-\int_{\gamma}\int_{\gamma}\xi(\mathbf{y})\mathbf{p}(x)\nabla_{\mathbf{x}}\bra{\mathbf{n}(\mathbf{y})\cdot\nabla\tilde{G}_{\kappa}\bra{\mathbf{y}-\mathbf{x}}}\dif\sigma(\mathbf{y})\dif\sigma(\mathbf{x})\\
    &=\hly{\int_{\gamma}\int_{\gamma}\xi(\mathbf{y})\mathbf{p}(x)\bra{\nabla^2\tilde{G}_{\kappa}(\mathbf{x}-\mathbf{y})\mathbf{n}(\mathbf{y})}\dif\sigma(\mathbf{y})\dif\sigma(\mathbf{x}),}
  \end{align*}
  where we have remembered that the tangential divergence defined in Section \ref{sec: classical traces} was adjoint to the negative surface gradient. Recognizing the Helmholtz vector single-layer potential in the first expression on the right hand side concludes the proof.
\end{proof}

\begin{proposition}
  \label{prop: adjointness}
    There exists a compact operator $\mathcal{C}_k:\mathcal{H}_N\rightarrow \mathcal{H}_N$ such that 
    \begin{equation*}
      \ip{\mathbb{A}^{NN}_\kappa(\vec{\mathbf{p}})}{\vec{\bm{\eta}}} = -\ip{\vec{\mathbf{p}}}{\mathbb{A}^{DD}_{\kappa}(\vec{\bm{\eta}})} + \ip{\mathcal{C}_k\vec{\mathbf{p}}}{\vec{\bm{\eta}}}
    \end{equation*}
    for all $\vec{\bm{\eta}}:=\bra{\bm{\eta},\,\, \xi}^{\top}\in\mathcal{H}_D$ and
    $\vec{\mathbf{p}}:=(\mathbf{p},\,\, q)^{\top}\in\mathcal{H}_N$.
  \end{proposition}
\begin{proof}
  Recall that $\mathbb{A}^{NN}_{\kappa} = \{\mathcal{T}_N\}\cdot\mathcal{S}\mathcal{L}_{\kappa}$ and $\mathbb{A}^{DD}_{\kappa} = \{\mathcal{T}_D\}\cdot\mathcal{D}\mathcal{L}_{\kappa}$. Since $\mathbf{curl}\circ\nabla =0$,
  $\langle\{\gamma_R\}\nabla\psi_{\tilde{k}}(q),\bm{\eta}\rangle_{\tau} = 0$ and $\langle\{\gamma_R\}\nabla\tilde{\psi}_{k}(\text{div}_{\Gamma}(\mathbf{p})),\bm{\eta}\rangle_{\tau} = 0$; therefore,
  \begin{align}
    \ip{\{\mathcal{T}_{N}\}\cdot\mathcal{SL}_k(\vec{\mathbf{p}})}{\vec{\bm{\eta}}}
    &= \langle-\{\gamma_R\}\bm{\Psi}_{\kappa}(\mathbf{p}),\bm{\eta}\rangle_{\tau} + \langle\{\gamma_n\}\nabla\psi_{\tilde{\kappa}}(q), \xi\rangle_{\Gamma}\nonumber\\
    &\qquad- \langle\{\gamma_n\}\bm{\Psi}_{\kappa}(\mathbf{p}),\xi \rangle_{\Gamma}  -\langle\{\gamma_n\}\nabla\tilde{\psi}_{\kappa}(\text{div}_{\Gamma}(\mathbf{p})), \xi\rangle_{\Gamma}.\label{eq: average single layer}
  \end{align}
  Since $\text{div}\circ\mathbf{curl}=0$, we also have $\{\gamma_D\}\,\mathbf{curl}\bm{\Psi_{\kappa}} =0$. Hence,  we need to compare \eqref{eq: average single layer} with 
  \begin{align*}
    \ip{\vec{\mathbf{p}}}{\{\mathcal{T}_{D}\}\cdot\mathcal{DL}_k(\vec{\bm{\eta}})}&=\ip{\mathbf{p}}{\{\gamma_t\}\mathbf{curl}\bm{\Psi}_k(\bm{\eta}\times\mathbf{n})}_{\tau} + \ip{q}{\{\eta\,\gamma_D\}\Upsilon_{\kappa}(\xi)}_{\Gamma}\\ 
    &\qquad+ \ip{\mathbf{p}}{\{\gamma_t\}\Upsilon_{\kappa}(\xi)}_{\tau}.
  \end{align*}
  The desired result follows by combining the known symmetry result from \eqref{eq: adjoint maxwell potential} with Lemma \ref{lem: grad phi term is anti hermitian} and Lemma \ref{lem: symmetry of two last terms single layer}.
\end{proof}

As consequence of Proposition \ref{prop: adjointness}, we have \begin{equation*}
    \bra{\mathbb{P}^+_{\kappa}}_{11}^* \hat{=}\bra{\mathbb{P}^-_{\kappa}}_{22},
\end{equation*} 
where $\hat{=}$ is used to indicate equality up to compact terms.

\section{Coupled problem}\label{sec: coupled problem}
\par In this section, we derive a variational formulation for the system \eqref{eq: strong from coupled sys a}-\eqref{eq: strong from coupled sys d} which couples a mixed variational formulation defined in the interior domain to a boundary integral equation of the first kind that arises in the exterior domain.

As proposed in \cite{arnold2006finite}, we introduce a new variable $P= -\text{div}\bra{\epsilon(\mathbf{x})\mathbf{U}}$ into equation \eqref{eq: strong from coupled sys a} to dispense with trial spaces contained in $\mathbf{H}(\mathbf{curl},\Omega_s)\cap\mathbf{H}(\text{div},\Omega_s)$. Applying Green's formula \eqref{curl curl integral identity} in $\Omega_s$, we obtain
\begin{align}
  \begin{split}
    \int_{\Omega_s}\mu^{-1}\,\mathbf{curl}\,\mathbf{U}\cdot\mathbf{curl}\,\mathbf{V}\dif\mathbf{x} + \int_{\Omega_s}\epsilon\,\nabla P\cdot\mathbf{V}\dif\mathbf{x} \qquad {} & \\ -\omega^2\int_{\Omega_s}\epsilon\,\mathbf{U}\cdot\mathbf{V}\dif\mathbf{x} 
    +\langle \gamma_{R,\mu}^{-}\mathbf{U},\gamma^-_t\mathbf{V}\rangle_{\tau} &= \bra{\mathbf{J},\mathbf{V}}_{\Omega_s},\\
    \int_{\Omega_s}P\,Q\dif\mathbf{x} - \int_{\Omega_s}\epsilon\,\mathbf{U}\cdot\nabla Q \dif\mathbf{x} +\langle \gamma^{-}_{n,\epsilon}\mathbf{U},\gamma^{-} Q\rangle_{\Gamma}&= 0
    \label{eq: green  mixed}
  \end{split}
\end{align}              
for all $\mathbf{V}\in\mathbf{H}\bra{\mathbf{curl}, \Omega_s}$, $Q\in H^1(\Omega_s)$. The volume integrals in these equations enter the interior symmetric bi-linear form
\begin{multline}\label{eq: sym bilinear form}
  \mathfrak{B}_{\kappa}\bra{\colvec{2}{\mathbf{U}}{P},\colvec{2}{\mathbf{V}}{Q}} :=   \int_{\Omega_s}\mu^{-1}\,\mathbf{curl}\,\mathbf{U}\cdot\mathbf{curl}\,\mathbf{V}\dif\mathbf{x} + \int_{\Omega_s}\epsilon\,\nabla P\cdot\mathbf{V}\dif\mathbf{x}\\
  + \int_{\Omega_s}P\,Q\dif\mathbf{x}
  - \int_{\Omega_s}\epsilon\,\mathbf{U}\cdot\nabla Q \dif\mathbf{x} -\omega^2\int_{\Omega_s}\epsilon\,\mathbf{U}\cdot\mathbf{V}\dif\mathbf{x}
\end{multline}
related to the one supplied for the Hodge-Laplace operator in \cite[Sec. 3.2]{arnold2010finite}. We aim to couple \eqref{eq: sym bilinear form} with the BIEs replacing the PDEs in $\Omega'$. We use the transmission conditions \eqref{eq: strong from coupled sys c}-\eqref{eq: strong from coupled sys d} to couple \eqref{eq: green  mixed} to the variational equation
\begin{equation*}
  \mathfrak{B}_{\kappa}\bra{\colvec{2}{\mathbf{U}}{P},\colvec{2}{\mathbf{V}}{Q}} + \Big\langle \mathcal{T}^+_{N} (\mathbf{U}^{\text{ext}}),\colvec{2}{\gamma^{-}_t\mathbf{V}}{\gamma^{-} Q}\Big\rangle = \mathscr{G}\bra{\colvec{2}{\mathbf{V}}{Q}}, \label{eq: bilinear var. 1}
\end{equation*}
which involves a functional 
\begin{equation*}
  \mathscr{G}\bra{(\mathbf{V}\, Q)^{\top}} := (\mathbf{J},\mathbf{V})_{\Omega_s} - \langle (\mathbf{g}_R\,g_n)^{\top},(\gamma^{-}_t\mathbf{V}\,\gamma^{-} Q)^{\top}\rangle
\end{equation*}
bounded over the test space. The exterior Calder\'on projector can be used to express the so-called Dirichlet-to-Neumann operator in two different ways.

1. Introducing the jump conditions into the \emph{first exterior Calder\'on identity} given on the first line of \eqref{exterior calderon} along with a new unknown $\vec{\mathbf{p}}=\mathcal{T}^+_{N} (\mathbf{U}^{\text{ext}})$ yields a variational system
\begin{equation}
\begin{aligned}
    \mathfrak{B}_{\kappa}\bra{\colvec{2}{\mathbf{U}}{P},\colvec{2}{\mathbf{V}}{Q}} + \Big\langle \vec{\mathbf{p}},\colvec{2}{\gamma^{-}_t\mathbf{V}}{\gamma^{-} Q} \Big\rangle &= \mathscr{G}\bra{\colvec{2}{\mathbf{V}}{Q}},\\
    \Big\langle \bra{\{\mathcal{T_{D}}\}\cdot\mathcal{DL}_{\kappa}+\frac{1}{2}\id}\mathcal{T}_{D,\epsilon}^-(\mathbf{U}), \vec{\mathbf{a}}\Big\rangle + \Big\langle \{\mathcal{T}_{D}\}\cdot\mathcal{SL}_{\kappa}\bra{\vec{\mathbf{p}}},\vec{\mathbf{a}}\Big\rangle &= \mathscr{R}\bra{\vec{\mathbf{a}}},
    \end{aligned}
    \label{eq:Johnson-Nedelec}
\end{equation}
  for all $(\mathbf{V}\,Q)^{\top}\in \mathbf{H}\bra{\mathbf{curl}, \Omega_s}\times H^1(\Omega_s)$ and $\vec{\mathbf{a}}\in\mathcal{H}_N$, resembling the original Johnson-Ned\'elec coupling \cite{aurada2013classical}. The new functional appearing on the right hand side of \eqref{eq:Johnson-Nedelec} is defined as 
  \begin{equation}
    \mathscr{R}\bra{\vec{\mathbf{a}}}: = \langle \bra{\{\mathcal{T}_{D}\}\cdot\mathcal{DL}_{\kappa}+\frac{1}{2}\id}(\bm{\zeta}_t,\,\zeta_D)^{\top}, \vec{\mathbf{a}}\rangle. 
  \end{equation}

2. Following the exposition of Costabel in \cite{costabel1987symmetric}, we also retain the \emph{second exterior Calder\'on identity} ---in which we again introduce the jump conditions to eliminate the dependence on the exterior solution--- and insert the resulting equation in \eqref{eq:Johnson-Nedelec} to obtain the symmetrically coupled problem. Again, the right hand side of our system of equations has to be modified to include a new bounded linear functional 
  \begin{equation}\mathscr{F}(\vec{\mathbf{V}}):= \mathscr{G}(V) + \langle -\{\mathcal{T}_{N}\}\cdot\mathcal{DL}_{\kappa}(\bm{\zeta}_t,\zeta_D)^{\top},(\gamma^{-}_t\mathbf{V},\,\gamma^{-} Q)^{\top}\rangle.
  \end{equation}
 
 We arrive at the following variational problem.
\begin{greyFrameCoupledProblem}
  Find $\vec{\mathbf{U}}:=(\mathbf{U},\,\,P)^{\top}\in\mathbf{H}\bra{\mathbf{curl}, \Omega_s}\times H^1(\Omega_s)$ and $\vec{\mathbf{p}}\in \mathcal{H}_{N}$ such that
  \begin{align}
    \label{Calderon coupled problem}
    \begin{split}
      \mathfrak{B}_{\kappa}\bra{\vec{\mathbf{U}},\vec{\mathbf{V}}} + \Big\langle \bra{-\mathbb{A}_{\kappa}^{NN}+\frac{1}{2}\id}\vec{\mathbf{p}},\colvec{2}{\gamma^{-}_t\mathbf{V}}{\gamma^{-} Q} \Big\rangle \qquad {}\qquad {}\qquad {}&\\
      +\,\Big\langle-\mathbb{A}_{\kappa}^{DN}\colvec{2}{\gamma_t^-\mathbf{U}}{-\gamma^-\bra{P}},\colvec{2}{\gamma^{-}_t\mathbf{V}}{\gamma^{-} Q}\Big\rangle &= \mathscr{F}\bra{\vec{\mathbf{V}}}\\
      \Big\langle \bra{\mathbb{A}^{DD}_{\kappa}+\frac{1}{2}\id}\colvec{2}{\gamma_t^-\mathbf{U}}{-\gamma^-\bra{P}}, \vec{\mathbf{a}}\Big\rangle + \Big\langle \mathbb{A}^{DD}_{\kappa}\bra{\vec{\mathbf{p}}},\vec{\mathbf{a}}\Big\rangle &= \mathscr{R}\bra{\vec{\mathbf{a}}}, 
    \end{split}
  \end{align}
  for all $\vec{\mathbf{V}}:=(\mathbf{V},\,\,Q)^{\top}\in\mathbf{H}\bra{\mathbf{curl}, \Omega_s}\times H^1(\Omega_s)$, $\vec{\mathbf{a}}\in \mathcal{H}_N$.
\end{greyFrameCoupledProblem}

\begin{remark} Part of the justification for using mixed formulations for problems involving the Hodge--Helmholtz/Laplace operator is the need to avoid trial spaces contained in $\mathbf{H}(\mathbf{curl},\Omega_s)\cap\mathbf{H}(\text{div},\Omega_s)$, because the latter doesn't allow for viable discretizations using finite elements \cite{arnold2010finite}. While from \eqref{eq:Johnson-Nedelec} the issue seems to reappear after using the Cald\'eron identities, the benefits of the introduced new unknown $P\in H^1(\Omega_s)$ in the mixed formulation conveniently carries over to the coupled system \eqref{Calderon coupled problem} upon substituting $-\gamma^-\bra{P}$ in place of $\gamma_{D,\epsilon}(\mathbf{U})$ in $\mathcal{T}^-_{D,\epsilon}(\mathbf{U})$.
\end{remark}

In the following proposition, we call \emph{forbidden resonant frequencies} the interior ``Dirichlet" (or electric) eigenvalues of the scaled Hodge-Laplace operator with constant coefficient $\eta = \mu_0\epsilon_0^2$. That is, $\kappa^2$ is a forbidden frequency if there exists a \emph{non-trivial} solution $\mathbf{U}\neq 0$ in $\mathbf{X}(\Delta,\Omega)$ to
\begin{align*}
    \Delta_{\eta}\mathbf{U}-\kappa^2\mathbf{U}&= 0,&&\text{in } \Omega_s,\\
    \mathcal{T}^{-}_D\mathbf{U}&=0, &&\text{on } \Gamma.
\end{align*}
We refer the reader to \cite{claeys2017first}, where the spectrum of the scaled Hodge-Laplace operator is completely characterized. See for e.g. \cite{schulz2020b}, \cite{sauter2010boundary}, \cite{christiansen2004discrete}, \cite{demkowicz1994asymptotic} and \cite{colton2013integral} for an overview of the issue of spurious resonances in electromagnetic and acoustic scattering models based on integral equations.
\begin{proposition}\label{prop: variational system solves transmission system}
  Suppose that $\kappa^2\in\mathbb{C}$ avoids forbidden resonant frequencies. By retaining an interior solution $U\in\mathbf{H}\bra{\mathbf{curl}, \Omega_s}$ and producing $\mathbf{U}^{\text{ext}}\in\mathbf{X}_{\text{\emph{loc}}}(\Delta,\Omega')$ using the representation formula \eqref{representation formula} for the obtained Cauchy data $(\vec{\mathbf{p}},\mathcal{T}_{D,\epsilon}^-U - (\bm{\zeta}_t,\,\,\zeta_D)^{\top})$ with $\gamma^-_{D,\epsilon}(\mathbf{U}) = -\gamma^-\bra{P}$, a solution to \eqref{Calderon coupled problem} solves the transmission system \eqref{eq: strong from coupled sys a}-\eqref{eq: strong from coupled sys d} in the sense of distribution.
\end{proposition}
\begin{proof}
\par The proof follows the approach in \cite[Lem. 6.1]{hiptmair2003coupling}. Since $\mathscr{D}(\Omega_s)^3\times C^{\infty}_0(\Omega_s)$ is a subset of the volume test space, any solution to the problem \eqref{Calderon coupled problem} solves \eqref{eq: strong from coupled sys a} in $\Omega_s$ in the sense of distribution. It follows that \eqref{eq: green  mixed} holds for all admissible $\vec{\mathbf{V}}$, which reduces \eqref{Calderon coupled problem} to the variational system
\begin{align*}
    0&=\Big\langle \bra{\mathbb{A}^{DD}_{\kappa}+\frac{1}{2}\id}\vec{\bm{\xi}}, \vec{\bm{\eta}}\Big\rangle + \Big\langle \{\mathbb{A}^{ND}_{\kappa}\bra{\vec{\mathbf{p}}},\vec{\bm{\eta}}\Big\rangle  \\
    0 &=-\Big\langle \vec{\mathbf{q}},\colvec{2}{\gamma^{-}_t\mathbf{V}}{\gamma^{-} Q}\Big\rangle + \Big\langle \bra{-\mathbb{A}^{NN}_{\kappa}+\frac{1}{2}\id}\vec{\mathbf{p}},\colvec{2}{\gamma^{-}_t\mathbf{V}}{\gamma^{-} Q} \Big\rangle\\
    &\qquad\qquad
      -\, \Big\langle    \mathbb{A}^{DN}_{\kappa}(\vec{\bm{\xi}}),\colvec{2}{\gamma^{-}_t\mathbf{V}}{\gamma^{-} Q}\Big\rangle
\end{align*}
where $\vec{\mathbf{q}} := \mathcal{T}_{N,\mu}^-(\mathbf{U}) - (\mathbf{g}_R,\,\,g_n)^{\top}$ and $\vec{\bm{\xi}}:=\mathcal{T}^-_{D,\epsilon}(\mathbf{U})-(\bm{\zeta}_t,\,\,\zeta_D)^{\top}$. 

We recognize in the equivalent operator equation
\begin{equation}
\underbrace{\begin{pmatrix}
    \mathbb{A}^{NN}_{\kappa}+\frac{1}{2}\id & \mathbb{A}^{DN}_{\kappa}\\
    \mathbb{A}^{ND}_{\kappa} & \mathbb{A}^{DD}_{\kappa}+\frac{1}{2}\id
\end{pmatrix}}_{\mathbb{P}^-_{\kappa}}
\begin{pmatrix}
\vec{\mathbf{p}}\\
\vec{\bm{\xi}}
\end{pmatrix}
= \begin{pmatrix}
\vec{\mathbf{p}}-\vec{\mathbf{q}}\\
0
\end{pmatrix}
\label{op eq calderon}
\end{equation}
the interior Cald\'eron projector \eqref{interior calderon} whose image is the space of valid Cauchy data for the homogeneous (scaled) Hodge--Laplace/Helmholtz interior equation with constant coefficient $\eta$. In particular, $\vec{\mathbf{p}}-\vec{\mathbf{q}}=\mathcal{T}^-_N\bra{\tilde{\mathbf{U}}}$ for some vector-field $\tilde{\mathbf{U}}\in \mathbf{X}\bra{\Delta,\Omega_s}$ satisfying
\begin{align}\label{eq: hodge-laplace homogeneous dirichlet BVP}
\begin{split}
\Delta_{\eta} \tilde{\mathbf{U}} - \kappa^2 \tilde{\mathbf{U}}&= 0,  \qquad\qquad\qquad\text{in }\Omega_s\\
\mathcal{T}^-_D\bra{\tilde{\mathbf{U}}} &= 0, \qquad\qquad\qquad \text{on }\Gamma.
\end{split}
\end{align}

If $\kappa^2\neq 0$, we rely on the hypothesis that $\kappa^2$ doesn't belong to the set of forbidden resonant frequencies to guarantee injectivity of the above boundary value problem \cite[Sec. 3]{claeys2017first}\cite[Sec. 3]{hazard1996solution}. Otherwise, the second Betti number of $\Omega_s$ being zero implies that zero is not a Dirichlet eigenvalue \cite[Sec. 4.5.3]{arnold2018}. We conclude that $\tilde{\mathbf{U}}=0$ is the unique trivial solution to \eqref{eq: hodge-laplace homogeneous dirichlet BVP}. Therefore, for the right hand side of \eqref{op eq calderon} to exhibit valid Neumann data, it must be that $\vec{\mathbf{p}}=\vec{\mathbf{q}}$. 

Now, the null space of the interior Cald\'eron projector $\mathbb{P}^-_{\kappa}$ coincides with valid Cauchy data for the exterior boundary value problem \eqref{eq: strong from coupled sys b} complemented with the radiation conditions at infinity introduced in Section \ref{sec: Introduction}.
In particular $(\vec{\mathbf{p}},\,\,\vec{\bm{\xi}})^{\top}$ is valid Cauchy data for that exterior Hodge--Helmholtz or Hodge--Laplace problem and $\mathbf{U}^{\text{ext}}=\mathcal{SL}_{\kappa}\bra{\vec{\mathbf{p}}} + \mathcal{DL}_{\kappa}\bra{\vec{\bm{\xi}}}$ solves \eqref{eq: strong from coupled sys b} and \eqref{eq: strong from coupled sys d} by construction. The fact that $\vec{\mathbf{p}}=\mathcal{T}^+_{N}\bra{\mathbf{U}^{\text{ext}}}$ solves \eqref{eq: strong from coupled sys c} is confirmed by the earlier observation that $\vec{\mathbf{p}}=\vec{\mathbf{q}}$.
\end{proof}

\begin{corollary}\label{cor: existence and uniqueness}
  Suppose that $\kappa^2\in\mathbb{C}$ avoids forbidden resonant frequencies. A solution pair $\bra{\vec{\mathbf{U}},\,\vec{\mathbf{p}}}$ to the coupled problem \eqref{Calderon coupled problem} is unique.
\end{corollary}

\begin{remark}
  We show in \cite{schulz2020b}, where the kernel of the coupled problem is completely characterized, that when $\kappa^2$ happens to be a resonant frequency, the interior solution $\mathbf{U}$ remains unique. This is no longer true for $\vec{\mathbf{p}}$ however, which is in general only unique up to Neumann traces of interior Dirichlet eigenfunctions of $\Delta_{\eta}$ associated to the eigenvalue $\kappa^2$. Fortunately, this kernel vanishes under the exterior representation formula obtained from \eqref{representation formula}.
\end{remark}

\section{Space decompositions}\label{sec: space decomposition}
Using the classical Hodge decomposition, a general inf-sup condition for Hodge--Laplace problems posed on closed Hilbert complexes was derived in \cite{arnold2010finite}. However, as orthogonality won't be important, we rather opt for the enhanced regularity of the regular decomposition proposed in \cite{buffa2003galerkin} and \cite{claeys2017first}. There, a continuous projection $\mathsf{Z}:\mathbf{H}\bra{\mathbf{curl},\Omega_s}\rightarrow\mathbf{H}^1(\Omega_s)$ is defined such that $\ker\bra{\mathsf{Z}}=\ker\bra{\mathbf{curl}}\cap\mathbf{H}\bra{\mathbf{curl},\Omega_s}$ and  $\mathbf{curl}\bra{\mathsf{Z}(\mathbf{U})}=\mathbf{curl}\bra{\mathbf{U}}$. 
From Rellich's theorem, this operator is compact as a mapping $\mathsf{Z}:\mathbf{H}\bra{\mathbf{curl},\Omega_s}\rightarrow \mathbf{L}^2(\Omega_s)$. Therefore, a stable direct regular decomposition 
\begin{equation}
  \mathbf{H}\bra{\mathbf{curl},\Omega_s} = \mathbf{X}(\mathbf{curl},\Omega_s) \oplus \mathbf{N}\bra{\mathbf{curl},\Omega_s}.
\end{equation} 
is provided by defining the subspaces $\mathbf{X}(\mathbf{curl},\Omega_s):=\mathsf{Z}\bra{\mathbf{H}\bra{\mathbf{curl},\Omega_s}}$ and $\mathbf{N}\bra{\mathbf{curl},\Omega_s}:=\ker\bra{\mathbf{curl}}\cap\mathbf{H}\bra{\mathbf{curl},\Omega_s}$.

A decomposition with similar properties can be designed for the space $\mathbf{H}^{-1/2}\bra{\text{div}_{\Gamma},\Gamma}$ with a projection operator $\mathsf{Z}^{\Gamma}:\mathbf{H}^{-1/2}\bra{\text{div}_{\Gamma},\Gamma}\rightarrow\mathbf{H}^{1/2}_R(\Gamma)$ satisfying $\ker(\mathsf{Z}^{\Gamma})=\ker\bra{\text{div}_{\Gamma}}\cap\mathbf{H}^{-1/2}\bra{\text{div}_{\Gamma},\Gamma}$ and $\text{div}_{\Gamma}\bra{\mathsf{Z}^{\Gamma}(\mathbf{p})}=\text{div}_{\Gamma}\bra{\mathbf{p}}$.

As before, the extra regularity of the range, in this case provided by \cite[Lem. 3.2]{hiptmair2003coupling}, leads to compactness of the mapping
$\mathsf{Z}^{\Gamma}:\mathbf{H}^{-1/2}\bra{\text{\text{div}}_{\Gamma},\Gamma}\rightarrow\mathbf{H}^{-1/2}_R(\Gamma)$.

The subspaces $\mathbf{X}\bra{\text{div}_{\Gamma},\Gamma}:=\mathsf{Z}^{\Gamma}\bra{\mathbf{H}^{-1/2}\bra{\text{div}_{\Gamma},\Gamma}}$ and $\mathbf{N}\bra{\text{div}_{\Gamma},\Gamma}:=\ker\bra{\text{div}_{\Gamma}}\cap\mathbf{H}^{-1/2}\bra{\text{div}_{\Gamma},\Gamma}$ provide a stable direct regular decomposition
\begin{equation}
  \mathbf{H}^{-1/2}\bra{\text{div}_{\Gamma},\Gamma} = \mathbf{X}\bra{\text{div}_{\Gamma},\Gamma}\oplus \mathbf{N}\bra{\text{div}_{\Gamma},\Gamma}.
\end{equation}

In the following, we may simplify notation by using $\mathbf{U}^{\perp}:=\mathsf{Z}\mathbf{U}$, $\mathbf{p}^{\perp}:=\mathsf{Z}^{\Gamma}\mathbf{p}$, $\mathbf{U}^0:=\bra{\id-\mathsf{Z}}\mathbf{U}$ and
$\mathbf{p}^0:=\bra{\id-\mathsf{Z}^{\Gamma}}\mathbf{p}$.

A very useful property of this pair of decompositions is stated is shown in \cite[Lem. 8.1]{hiptmair2003coupling} and \cite[Lem. 8.2]{hiptmair2003coupling}:
  The operators
  \begin{subequations}
    \begin{equation}\label{eq: old lem 4.6 1}
      \bra{\gamma_t^{-}}'\circ\bra{\{\gamma_R\}\bm{\Psi}_{\kappa}+\frac{1}{2}\id}:\mathbf{N}\bra{\text{\text{div}}_{\Gamma},\Gamma}\rightarrow \mathbf{N}\bra{\mathbf{curl},\Omega_s}',
    \end{equation}
    and
    \begin{equation}\label{eq: old lem 4.6 2}
      \bra{\gamma_t^{-}}'\circ\bra{\{\gamma_R\}\bm{\Psi}_{\kappa}+\frac{1}{2}\id}:\mathbf{X}\bra{\text{\text{div}}_{\Gamma},\Gamma}\rightarrow \mathbf{X}\bra{\mathbf{curl},\Omega_s}'
    \end{equation}	
  \end{subequations}
are compact.

Another benefit of this pair of regular decompositions will become explicit in the poof of
Lemma \ref{lem: coercivity TNDL} found in the next
section.

It follows from \cite[Lem. 6.4]{claeys2017first} that
$\text{div}_{\Gamma}:\mathbf{X}\bra{\text{div}_{\Gamma},\Gamma}\rightarrow
H^{-1/2}_*(\Gamma)$ is a continuous bijection. The bounded inverse theorem guarantees the
existence of a continuous inverse
$\bra{\text{div}_{\Gamma}}^{\dag}:H^{-1/2}_*(\Gamma)\rightarrow
\mathbf{X}\bra{\text{div}_{\Gamma},\Gamma}$ such that
\begin{align*}
  \bra{\text{div}_{\Gamma}}^{\dag}\circ\text{div}_{\Gamma} &=\id\Big\vert_{\mathbf{X}\bra{\text{div}_{\Gamma},\Gamma}}, & &\text{div}_{\Gamma}\circ\bra{\text{div}_{\Gamma}}^{\dag}=\id\Big\vert_{H^{-1/2}_*(\Gamma)}.
\end{align*} 

\section{Well-posedness of the coupled variational problem}\label{sec: compactness and coercivity}
We use the direct decompositions introduced in Section \ref{sec: space decomposition} to prove that the bilinear form associated to the coupled system \eqref{prop: variational system solves transmission system} of Section \ref{sec: coupled problem} satisfies a generalized G{\aa}rding inequality. 

The coupled variational problem \eqref{Calderon coupled problem} translates into the operator equation 
\begin{equation*}
  \mathbb{G}_{\kappa}
  \begin{pmatrix}
    \vec{\mathbf{U}}\\
    \vec{\mathbf{p}}
  \end{pmatrix} =
  \colvec{2}{\mathscr{F}}{\mathscr{R}}\in \bra{\mathbf{H}\bra{\mathbf{curl},\Omega_s}\times H^1(\Omega)}'\times \bra{\mathcal{H}_N}'.
  \quad
\end{equation*}

Letting $\mathsf{B}_{\kappa}:\mathbf{H}\bra{\mathbf{curl},\Omega_s}\times H^1\bra{\Omega_s}\rightarrow\bra{\mathbf{H}\bra{\mathbf{curl},\Omega_s}\times H^1\bra{\Omega_s}}'$ be the operator
\begin{equation*}\langle\mathsf{B}_{\kappa}\bra{\vec{\mathbf{U}}}\vec{\mathbf{V}}\rangle:=\mathfrak{B}_{\kappa}\bra{\vec{\mathbf{U}},\vec{\mathbf{V}}}
\end{equation*} 
associated with the Hodge--Helmholtz/Laplace volume contribution to the system, the operator 
\begin{equation*}
\mathbb{G}_{\kappa}:\bra{\mathbf{H}\bra{\mathbf{curl},\Omega_s}\times H^1(\Omega)}\times \mathcal{H}_N\rightarrow\bra{\mathbf{H}\bra{\mathbf{curl},\Omega_s}\times H^1(\Omega)}'\times \bra{\mathcal{H}_N}'
\end{equation*}
can be represented by the block operator matrix
  \begin{equation*}\label{eq: coupled equation operator matrix}
    \mathbb{G}_{\kappa}=\begin{pmatrix}
      \begin{array}{c|c}
	{\color{violet}\mathsf{B}_{\kappa}}{\color{purple}-\colvec{2}{\bra{\gamma_t^-}'}{\bra{\gamma^-}'}\cdot\mathbb{A}^{DN}_{\kappa}\cdot\colvec{2}{\gamma_t^-}{-\gamma^-}} & {\color{olive}\colvec{2}{\bra{\gamma_t^-}'}{\bra{\gamma^-}'}\cdot\bra{\mathbb{P}^+_{\kappa}}_{22}}\\
	\hline
	{\color{teal}\bra{\mathbb{P}_{\kappa}^-}_{11}\cdot\colvec{2}{\gamma_t^-}{-\gamma^-}} & {\color{magenta}\mathbb{A}_{\kappa}^{ND}}
      \end{array}
    \end{pmatrix},
  \end{equation*}
shown here in ``variational arrangement".

The symmetry revealed in \Cref{sec: Integral operators} makes explicit much of the structure of the above operator. We have introduced colors to better highlight the contribution of each individual block in the following sections.

Our goal is to design an isomorphism $\mathbb{X}$ of the test space and resort to compact perturbations of $\mathbb{G}_{\kappa}\circ\mathbb{X}^{-1}$ to achieve an operator block structure with diagonal blocks that are elliptic over the splittings of Section \ref{sec: space decomposition} and off-diagonal blocks that fit a skew-symmetric pattern. Stability of the coupled system can then be obtained from the next theorem. An overline indicates component-wise complex conjugation.
\begin{theorem}[{\cite[Thm. 4]{buffa2003galerkin}}]\label{thm: coercivity implies fredholm}
	If a bilinear form $a:V\times V\rightarrow \mathbb{C}$ on a reflexive Banach space $V$ is T-coercive:
	\begin{equation}\label{eq: t-coercivity def}
		\abs{a\bra{u,\mathbb{X}\overline{u}} + c\bra{u,\overline{u}}}\geq C\norm{u}^2_V\quad\forall u\in V,
	\end{equation}
with $C>0$, $c:V\times V\rightarrow\mathbb{C}$ compact and $\mathbb{X}:V\rightarrow V$ an isomorphism of $V$, then the operator $A:V\rightarrow V'$ defined by $A:u\mapsto a(u,\cdot)$ is Fredholm with index 0. 
\end{theorem}

The authors of \cite{Buffa2002Boundary} refer to \eqref{eq: t-coercivity def} as ``Generalized G{\aa}rding inequality", because
\begin{equation*}
\abs{a\bra{u,\mathbb{X}\overline{u}}}\geq C\norm{u}^2_V -\abs{c\bra{u,\overline{u}}}\qquad \forall\,u\in V,
\end{equation*}
generalizes the classical G{\aa}rding inequality for a bilinear form $b$ associated with uniformly elliptic operator of even order $2\ell$: $\exists\, C_2\geq 0, C_1>0$ such that
\begin{equation*}
b(u,u) \geq C_1\norm{u}^2_{H^\ell(\Omega)}- C_2\norm{u}_{L^2\bra{\Omega}}\qquad \forall\,u\in H^\ell_0(\Omega).
\end{equation*}
Assuming that \eqref{eq: t-coercivity def} holds with $\mathbb{X}=\id$, a simple proof of the stability estimate
$
\norm{u}_V \leq C\norm{f}_{V'}
$,
obtained for the unique solution of the operator equation $Au=f$ when $A$ is injective is given in \cite[Thm. 3.15]{steinbach2007numerical}. A proof of the general case can be deduced from \cite{hildebrandt1964constructive}. T-coercivity theory is a reformulation of the Banach-Ne{\u c}as-Babu{\u s}ka theory. The former relies on the construction of explicit inf-sup operators at the discrete and continuous levels, whereas the later develops on an abstract inf-sup condition \cite{ciarlet2012t}.

In deriving the following results, it will be convenient to denote $\vec{\mathbf{U}} := (\mathbf{U},\,\,P)^{\top}\in \mathbf{H}\bra{\mathbf{curl},\Omega_s}\times H^1(\Omega)$ and $\vec{\mathbf{p}} := (\mathbf{p},\,\,q)^{\top}\in\mathcal{H}_N$. We indicate with a hat equality up to a compact perturbation (e.g. $\hat{=}$).

\subsection{Space isomorphisms}
In this section, we take up the challenge of finding a suitable isomorphism $\mathbb{X}$. We build it separately for the function spaces in $\Omega_s$ and on the boundary $\Gamma$. Crucial hints are offered by the construction of the sign-flip isomorphism for the classical electric wave equation in \cite{buffa2003galerkin}.

We start with devising an isomorphism $\Xi$ of the volume function spaces and show that the upper-left diagonal block of $\mathbb{G}_{\kappa}$ satisfy a generalized G{\aa}rding inequality.

Under the assumption that the first Betti number of $\Omega_s$ is zero, there exists a bijective ``scalar potential lifting" 
$
\mathsf{S}:\mathbf{N}(\mathbf{curl},\Omega_s)\rightarrow H^1_*\bra{\Omega_s}
$
satisfying $\nabla \mathsf{S}\bra{\mathbf{U}}=\mathbf{U}$. The Poincar\'e-Friedrichs inequality guarantees that this map is continuous.

Notice that since it also follows from the Poincar\'e-Friedrichs inequality that $\nabla:H_*^1(\Omega_s)\rightarrow\mathbf{N}(\mathbf{curl},\Omega_s)$ is injective, $\mathsf{S}\circ\nabla:H^1(\Omega_s)\rightarrow H_*^1(\Omega_s)$ is a bounded projection onto the space of Lebesgue measurable functions having zero mean. Its nullspace consists of the constant functions in $\Omega_s$.

\begin{proposition}\label{prop: isomorphism volume space}
  For any $\theta>0$ and $\beta>0$, the bounded linear operator $\Xi:\mathbf{H}\bra{\mathbf{curl}, \Omega_s}\times H^1(\Omega_s)\rightarrow\mathbf{H}\bra{\mathbf{curl}, \Omega_s}\times H^1(\Omega_s)$ defined by
  \begin{equation*}
    \Xi\bra{\vec{\mathbf{U}}}:=\colvec{2}{\mathbf{U}^{\perp}-\mathbf{U}^0+\beta\,\nabla P}{-\theta\bra{\mathsf{S}\bra{\mathbf{U}^0}+\,\beta\,\mathbf{mean}\bra{P}}},\qquad\qquad \vec{\mathbf{U}} = (\mathbf{U},\,\,P)^{\top},
  \end{equation*}
  has a continuous inverse. In other words, $\Xi$ is an isomorphism of Banach spaces.
\end{proposition}
\begin{proof}
  By showing that $\Xi$ is a bijection, the theorem follows as a consequence of the bounded inverse theorem.
  
  Let $\bra{\mathbf{V}\,\,Q}^{\top}\in \mathbf{H}\bra{\mathbf{curl},  \Omega_s}\times H^1(\Omega_s)$. Since $\nabla Q\in\mathbf{N}\bra{\mathbf{curl},\Omega_s}$, we immediately have $\mathsf{Z}\bra{\mathbf{V}^{\perp}-\theta^{-1}\nabla Q}=\mathbf{V}^{\perp}$ and $\bra{\id-\mathsf{Z}}\bra{\mathbf{V}^{\perp}-\theta^{-1}\nabla Q}= -\theta^{-1}\nabla Q$. Hence, relying on the resulting observation that 
  \begin{equation*}
    \nabla \mathsf{S}\bra{\bra{\mathbf{V}^{\perp}-\theta^{-1}\nabla Q}^0} = -\theta^{-1}\nabla Q
  \end{equation*} and exploiting that $\mathbf{mean}\bra{H^1_*(\Omega_s)}=\{0\}$, we have
  \begin{align}\label{eq: inverse of Xi candidate}
    \Xi\bra{\colvec{2}{\mathbf{V}^{\perp}-\theta^{-1}\nabla Q}{\beta^{-1}\bra{\mathsf{S}\bra{\mathbf{V}^0}-\theta^{-1}Q}}} 
    &=\colvec{2}{\mathbf{V}}{\mathsf{S}\bra{\nabla Q}+\mathbf{mean}\bra{Q}}.
  \end{align} 
  Since $H^1(\Omega_s)$ decomposes into the stable direct sum of $H^1_*(\Omega_s)$ and the space of constant functions in $\Omega_s$, \eqref{eq: inverse of Xi candidate} shows that $\Xi$ is surjective.
  
  Now, suppose that $\Xi\bra{\vec{\mathbf{V}}}=\Xi\bra{\vec{\mathbf{U}}}$. Then, we have
  \begin{equation*}
    \mathbf{U}^0-\mathbf{V}^0 = \nabla \mathsf{S}\bra{\mathbf{U}^0-\mathbf{V}^0} = \beta\,\nabla\bra{\mathbf{mean}\bra{Q-P}} = 0.
  \end{equation*}
  Since the considerations of Section \ref{sec: space decomposition} readily yield that $\mathbf{V}^\perp = \mathbf{U}^\perp$, we conclude that $\mathbf{V}=\mathbf{U}$. In turn, it follows that $\nabla P= \nabla Q$ and $\mathbf{mean}(P)=\mathbf{mean}(Q)$. Therefore, $\Xi$ is injective.
\end{proof}

We now turn to the design of an isomorphism for the Neumann trace space $\mathcal{H}_N$ and prove that the lower-right block $\mathbb{A}_{\kappa}^{ND}$ of $\mathbb{G}_{\kappa}$ satisfies a generalized G{\aa}rding inequality. 
\begin{proposition}\label{prop: isomorphism trace space}
  For any $\tau > 0$ and $\lambda>0$, the bounded linear operator $\Xi^{\Gamma}:\mathcal{H}_N\rightarrow\mathcal{H}_N$ defined by
  \begin{equation*}
    \Xi^{\Gamma}(\vec{\mathbf{p}}) :=\colvec{2}{\mathbf{p}^{\perp}-\mathbf{p}^0 -\lambda\bra{\text{\emph{div}}_{\Gamma}}^{\dag}\mathsf{Q}_*q}{-\tau\bra{\text{\emph{div}}_{\Gamma}\bra{\mathbf{p}}+\lambda\,\mathbf{mean}\bra{q}}}, \qquad\qquad \vec{\mathbf{p}} = (\mathbf{p},\,\,q)^{\top},
  \end{equation*}
  has a continuous inverse. In other words, $\Xi^{\Gamma}$ is an isomorphism of Banach spaces.
\end{proposition}
\begin{proof}
    We proceed as in proposition \ref{prop: isomorphism volume space}. Since $\bra{\text{div}_{\Gamma}}^{\dag}\mathsf{Q}_*q\in\mathbf{X}(\text{div}_{\Gamma},\Gamma)$, we have $\mathsf{Z}^{\Gamma}\bra{\Xi_1^{\Gamma}(\vec{\mathbf{p}})}=\mathbf{p}^{\perp}-\bra{\text{div}_{\Gamma}}^{\dag}\mathsf{Q}_*q$. Using that $\mathbf{mean}\circ\text{div}_{\Gamma}=0$ and $\bra{\text{div}_{\Gamma}}^{\dag}\text{div}_{\Gamma}\mathbf{p}=\mathbf{p}^{\perp}$, we evaluate
	\begin{align*}
	\Xi^{\Gamma}\bra{\colvec{2}{-\mathbf{p}^0-\tau^{-1}\bra{\text{div}_{\Gamma}}^{\dag}\mathsf{Q}_*q}{\lambda^{-1}\bra{-\text{div}_{\Gamma}(\mathbf{p})-\tau^{-1}q}}} &=\colvec{2}{\mathbf{p}^0+\mathbf{p}^{\perp}}{\mathsf{Q}_*q+\mathbf{mean}\bra{q}}.
	\end{align*}
	This shows that $\Xi^{\Gamma}$ is surjective.
	
	Suppose that $X^{\Gamma}(\vec{\mathbf{p}})=X^{\Gamma}(\vec{\mathbf{a}})$. It is immediate that $\mathbf{p}^0=\mathbf{a}^0$. On the one hand, we obtain from $X_1^{\Gamma}(\vec{\mathbf{p}})=X_1^{\Gamma}(\vec{\mathbf{a}})$ that
	\begin{equation}\label{eq: injectivity Xi gamma proof obs 1}
	\mathbf{p}^{\perp}-\mathbf{a}^{\perp} = \lambda\bra{\text{div}_{\Gamma}}^{\dag}\bra{\mathsf{Q}_*q -\mathsf{Q}_*b}.
	\end{equation}
	On the other hand, $X_2^{\Gamma}(\vec{\mathbf{p}})=X_2^{\Gamma}(\vec{\mathbf{a}})$ implies that
	\begin{equation}\label{eq: injectivity Xi gamma proof obs 2}
	\text{div}_{\Gamma}\bra{\mathbf{p}-\mathbf{a}} = \lambda\,\mathbf{mean}\bra{q-b}.
	\end{equation}
	Relying on the fact that $\text{div}_{\Gamma} = \text{div}_{\Gamma}\circ\mathsf{Z}^{\Gamma}$ again, combining \eqref{eq: injectivity Xi gamma proof obs 1} and \eqref{eq: injectivity Xi gamma proof obs 2} yields
	\begin{equation*}
	\mathsf{Q}_*q + \mathbf{mean}(q) = \mathsf{Q}_*b + \mathbf{mean}(b).
	\end{equation*}
	Evidently, \eqref{eq: injectivity Xi gamma proof obs 1} then also guarantees that $\mathbf{p}^{\perp}=\mathbf{a}^{\perp}$. We can finally conclude that $X^{\Gamma}$ is injective and thus the result follows from the bounded inverse theorem.
\end{proof}	

In the following, we will write $\Xi^{\Gamma}_1$ and $\Xi^{\Gamma}_2$ for the components of the isomorphism of the trace space.	

\subsection{Main result}
\label{sec:main}

The main result of this work, stated in Theorem \ref{thm: main theorem}, asserts that the
operator $\mathbb{G}_{\kappa}$ associated with the coupled system \eqref{Calderon coupled
  problem} is well-posed when $\kappa^2$ lies outside the discrete set of forbidden
frequencies described in \cite{claeys2017first}. It relies on two propositions, whose
proofs are postponed until the end of section \ref{sec: compactness and coercivity}.

The first claims that the \emph{block diagonal of} $\mathbb{G}_{\kappa}$ (as a sum of block operators) \emph{is T-coercive}.
\begin{proposition}\label{prop: coercivity volume and TNDL}
  For any frequency $\omega\geq0$, there exist a compact operator $\mathsf{K}:\mathbf{H}\bra{\mathbf{curl}, \Omega_s}\times H^1(\Omega_s)\times\mathcal{H}_N\rightarrow\mathbf{H}\bra{\mathbf{curl}, \Omega_s}\times H^1(\Omega_s)\times\mathcal{H}_N$, a positive constant $C>0$ and parameters $\theta>0$ and $\tau>0$, possibly depending on $\Omega_s$, $\epsilon$, $\mu$, $\kappa$ and $\omega$, such that
  \begin{multline*}
    \mathfrak{Re}\,\Bigg\langle \diag\bra{\mathbb{G}_{\kappa}}\colvec{2}{\vec{\mathbf{U}}}{\vec{\mathbf{p}}},\colvec{2}{\Xi\,\vec{\overline{\mathbf{U}}}}{\Xi^{\Gamma}\vec{\overline{\mathbf{p}}}}\Big\rangle + \Big\langle\mathsf{K}\colvec{2}{\vec{\mathbf{U}}}{\vec{\mathbf{p}}},\colvec{2}{\vec{\overline{\mathbf{U}}}}{\vec{\overline{\mathbf{p}}}}\Bigg\rangle\\
    \geq C \bra{\norm{\mathbf{U}}^2_{\mathbf{H}\bra{\mathbf{curl}, \Omega_s}} + \norm{P}^2_{H^1\bra{\Omega_s}}+\norm{\vec{\mathbf{p}}}^2_{\mathcal{H}_N}}
  \end{multline*}
  for all $\vec{\mathbf{U}}:=\bra{\mathbf{U}\,\,P}^{\top}\in\mathbf{H}\bra{\mathbf{curl}, \Omega_s}\times H^1(\Omega_s)$ and $\vec{\mathbf{p}}\in \mathcal{H}_N$.
\end{proposition}

The proof of this proposition will rely on several steps: Lemma \ref{lem: coercivity B}, Lemma \ref{lem: coercivity TNDL} and Lemma \ref{lem: coercivity of TDSL}.

The second proposition states that the off-diagonal blocks are compact operators. The proof of that fact relies on definitions and results that belong to the next technical section. It will materialize as the last piece of the puzzle that completes the proof of the T-coercivity of $\mathbb{G}_{\kappa}$.
\begin{proposition}\label{prop: compactness off-diagonal blocks}
  For any frequency $\omega\geq0$, there exists, for a suitable choice of $\tau$, $\beta$, $\theta$ and $\lambda$, a continuous compact endomorphism $\mathsf{K}$ of the space $\mathbf{H}\bra{\mathbf{curl}, \Omega_s}\times H^1(\Omega_s)\times\mathcal{H}_N$ such that
  \begin{equation}\label{eq: K kappa compact coercivity}
    \mathfrak{Re}\,\Bigg\langle \bra{\mathbb{G}_{\kappa}-\diag\bra{\mathbb{G}_{\kappa}}}\colvec{2}{\vec{\mathbf{U}}}{\vec{\mathbf{p}}},\colvec{2}{\Xi\,\vec{\overline{\mathbf{U}}}}{\Xi^{\Gamma}\vec{\overline{\mathbf{p}}}}\Bigg\rangle
    = \Big\langle \mathsf{K}\colvec{2}{\vec{\mathbf{U}}}{\vec{\mathbf{p}}},\colvec{2}{\vec{\overline{\mathbf{U}}}}{\vec{\overline{\mathbf{p}}}}\Big\rangle.
  \end{equation}
\end{proposition}

The main result immediately follows from the two previous propositions.
\begin{theorem}\label{thm: main theorem}
  For any $\omega\geq 0$, there exists an isomorphism $\mathbb{X}_{\kappa}$ of the trial space $\mathbf{H}\bra{\mathbf{curl}, \Omega_s}\times H^1(\Omega_s)\times\mathcal{H}_N$, and compact operator $\mathbb{K}:\mathbf{H}\bra{\mathbf{curl}, \Omega_s}\times H^1(\Omega_s)\times\mathcal{H}_N\rightarrow\bra{\mathbf{H}\bra{\mathbf{curl}, \Omega_s}\times H^1(\Omega_s)}'\times\mathcal{H}_N'$ such that
  \begin{equation*}
    \mathfrak{Re}\,\Bigg\langle \bra{\mathbb{G}_{\kappa}+\mathbb{K}}\colvec{2}{\vec{\mathbf{U}}}{\vec{\mathbf{p}}},\mathbb{X}\colvec{2}{\vec{\overline{\mathbf{U}}}}{\vec{\overline{\mathbf{p}}}}\Bigg\rangle\geq C\bra{\norm{\mathbf{U}}^2_{\mathbf{H}\bra{\mathbf{curl},\Omega_s}} + \norm{P}_{H^1(\Omega_s)}^2+\norm{\vec{\mathbf{p}}}^2_{\mathcal{H}_N}}
  \end{equation*}
  for some positive constant $C>0$.
\end{theorem}
\begin{proof}
  The proof will amount to the validation that the choices of parameters in the previous propositions \ref{prop: coercivity volume and TNDL} and \ref{prop: compactness off-diagonal blocks} are compatible.
\end{proof}

The following corollary is immediate upon applying \Cref{thm: coercivity implies fredholm}.
\begin{corollary}
  The system operator $\mathbb{G}_k:\mathbf{H}\bra{\mathbf{curl}, \Omega_s}\times H^1(\Omega_s)\times\mathcal{H}_N\rightarrow\bra{\mathbf{H}\bra{\mathbf{curl}, \Omega_s}\times H^1(\Omega_s)}'\times\mathcal{H}_N'$ associated with the variational problem \eqref{Calderon coupled problem} is Fredholm of index 0.
\end{corollary}

Injectivity, guaranteed when $\kappa^2$ avoids resonant frequencies by corollary \ref{cor: existence and uniqueness}, yields well-posedness.

\subsection{T-Coercivity of the diagonal blocks}
Equipped with the isomorphism $\Xi$, let us now study coercivity of the bilinear form $\mathfrak{B}_{\kappa}$ defined in \eqref{eq: sym bilinear form} and associated to the Hodge--Helmholtz/Laplace operator.
\begin{lemma}\label{lem: coercivity B}
  For any frequency $\omega\geq0$ and parameter $\beta>0$, there exist a positive constant $C>0$ and a parameter $\theta>0$, possibly depending on $\Omega_s$, $\mu$, $\epsilon$ and $\omega$, and a compact bounded sesqui-linear form $\mathfrak{K}$ defined over $\mathbf{H}\bra{\mathbf{curl}, \Omega_s}\times H^1(\Omega_s)$, such that
  \begin{equation*}
    \mathfrak{Re}\bra{{\color{violet}\mathfrak{B}_{\kappa}}\bra{\vec{\mathbf{U}},\Xi\,\vec{\overline{\mathbf{U}}}} - \mathfrak{K}\bra{\vec{\mathbf{U}},\vec{\mathbf{U}}}} \geq C \bra{\norm{\mathbf{U}}^2_{\mathbf{H}\bra{\mathbf{curl}, \Omega_s}} + \norm{P}^2_{H^1\bra{\Omega_s}}}
  \end{equation*}
  for all $\vec{\mathbf{U}}:=\bra{\mathbf{U},\,\,P}^{\top}\in\mathbf{H}\bra{\mathbf{curl}, \Omega_s}\times H^1(\Omega_s)$.
\end{lemma}
\begin{proof}
  As $\mathbf{curl}\bra{\mathbf{U}^0}=0$, $\mathbf{curl}\bra{\nabla P} =0$, and $\nabla \circ \mathbf{mean} = 0$, we evaluate
  \begin{align*}
    &\mathfrak{B}_{\kappa}\bra{\colvec{2}{\mathbf{U}}{P},\colvec{2}{\overline{\mathbf{U}}^{\perp}-\overline{\mathbf{U}}^0+\beta\overline{\nabla P}}{-\theta\bra{\mathsf{S}\bra{\overline{\mathbf{U}}^0}+\beta\,\mathbf{mean}\bra{\overline{P}}}}} \\
    &=\bra{\mu^{-1}\mathbf{curl}\bra{\mathbf{U}^{\perp}},\mathbf{curl}\bra{\mathbf{U}^{\perp}}}_{\Omega_s}
    +\bra{\epsilon\nabla P, \mathbf{U}^{\perp}}_{\Omega_s}
    -\bra{\epsilon\nabla P, \mathbf{U}^{0}}_{\Omega_s}\\ 
    &\qquad+ \beta\bra{\epsilon\nabla P, \nabla P}_{\Omega_s}  +\theta\bra{\epsilon\mathbf{U}^{\perp},\mathbf{U}^0}_{\Omega_s}+\theta\bra{\epsilon\mathbf{U}^{0},\mathbf{U}^0}_{\Omega_s}\\
    &\qquad-\omega^2\bra{\epsilon\mathbf{U}^{\perp},\mathbf{U}^{\perp}-\mathbf{U}^0+\beta\nabla P}_{\Omega_s}
    -\omega^2\bra{\epsilon\mathbf{U}^0,\mathbf{U}^{\perp}}+\omega^2\bra{\epsilon\mathbf{U}^0,\mathbf{U}^{0}}\\
    &\qquad-\beta\omega^2\bra{\epsilon\mathbf{U}^0,\nabla P}
    -\bra{P,\theta \mathsf{S}\bra{\mathbf{U}^0}}_{\Omega_s}-\bra{P,\theta \beta\,\mathbf{mean}(P)}_{\Omega_s}.
  \end{align*}
  Upon application of the Cauchy-Schwartz inequality, the bounded sesqui-linear form
  \begin{align*}
    \mathfrak{K}\bra{\vec{\mathbf{U}},\vec{\mathbf{U}}} &:= \bra{\epsilon\nabla P,
      \mathbf{U}^{\perp}}_{\Omega_s}-\bra{P,\theta
      \mathsf{S}\bra{\mathbf{U}^0}}_{\Omega_s}
    +\theta\bra{\epsilon\mathbf{U}^{\perp},\mathbf{U}^0}_{\Omega_s}
    \\
    &\qquad-\omega^2\bra{\epsilon\mathbf{U}^0,\mathbf{U}^{\perp}}_{\Omega_s}-\omega^2\bra{\epsilon\mathbf{U}^{\perp},\mathbf{U}^{\perp}-\mathbf{U}^0+\beta\nabla
      P}_{\Omega_s}\\
     &\qquad-\bra{P,\theta \beta\,\mathbf{mean}(P)}_{\Omega_s}
  \end{align*}
  is shown to be compact by compactness of $\mathsf{Z}$ and the Rellich theorem. Using Young's inequality twice with $\delta>0$, we estimate
  \begin{align*}
    &\mathfrak{Re}\bra{\mathfrak{B}_{\kappa}\bra{\vec{\mathbf{U}},\Xi\,\vec{\mathbf{U}}}
      -\mathfrak{K}\bra{\vec{\mathbf{U}},\vec{\mathbf{U}}}}\\
    &\qquad\geq\mu^{-1}_{\text{max}}\,\norm{\mathbf{curl}\,\mathbf{U}^{\perp}}^2_{\Omega_s}
    +
    \bra{\epsilon_{\text{min}}\bra{\theta+\omega^2}-\delta\,\epsilon_{\text{max}}\bra{1+\beta\omega^2}}\norm{\mathbf{U}^0}^2_{\Omega_s}
    \\
    &\qquad\qquad+\mathfrak{Re}\bra{\epsilon_{\text{min}}\,\beta -\frac{1}{\delta} \epsilon_{\text{max}}\,\bra{1 +\beta\omega^2}}\norm{\nabla P}^2_{\Omega_s}.
  \end{align*}
  The operator $\mathbf{curl}:\mathsf{Z}\bra{\mathbf{H}\bra{\mathbf{curl},\Omega}}\rightarrow \mathbf{L}^2\bra{\Omega_s}$ is a continuous injection, hence since its image is closed in $\mathbf{L}^2\bra{\Omega_s}$,  it is also bounded below. Hence, for any $\beta>0$, choose $\delta>0$ large enough, then $\theta>0$ accordingly large, and the desired inequality follows.
\end{proof}

The complex inner products
  \begin{align*}
    \bra{a,b}_{-1/2} &:=\int_{\Gamma}\int_{\Gamma}G_0\bra{\mathbf{x}-\mathbf{y}}a(\mathbf{x})\,\overline{b(\mathbf{y})}\dif\sigma(\mathbf{x})\dif\sigma(\mathbf{y}),\\
    \bra{\mathbf{a},\mathbf{b}}_{-1/2} &:=\int_{\Gamma}\int_{\Gamma}G_0\bra{\mathbf{x}-\mathbf{y}}\mathbf{a}(\mathbf{x})\cdot\overline{\mathbf{b}(\mathbf{y})}\dif\sigma(\mathbf{x})\dif\sigma(\mathbf{y}),
  \end{align*}
defined over $H^{-1/2}(\Gamma)$ and $\mathbf{H}^{-1/2}\bra{\text{div}_{\Gamma},\Gamma}$ respectively, are positive definite Hermitian forms and they induce equivalent norms on the trace spaces \cite[Sec. 4.1]{Buffa2002Boundary}. Combined with the stability of the decomposition introduced in Section \ref{sec: space decomposition}, this observation also allows us to conclude that 
\begin{equation*}
  \mathbf{a}\mapsto\norm{\text{div}_{\Gamma}\bra{\mathbf{a}}}_{-1/2} + \norm{(\id - P^{\Gamma})\,\mathbf{a}}_{-1/2}
\end{equation*} 
also defines an equivalent norm in $\mathbf{H}^{-1/2}\bra{\text{div}_{\Gamma},\Gamma}$.

Let us denote the two components of the isomorphism $\Xi$ by 
\begin{align*}
  \Xi_1(\vec{\mathbf{U}}):=\mathbf{U}^{\perp}-\mathbf{U}^0+\nabla P,\quad 
  \Xi_2(\vec{\mathbf{U}}):=-\theta\bra{\mathsf{S}\bra{\mathbf{U}^0}+\,\mathbf{mean}\bra{P}}.
\end{align*}
We now derive an estimate similar to the one found in Lemma \ref{lem: coercivity B} that
completes the proof of the coercivity of the upper-left diagonal block of
$\mathbb{G}_{\kappa}$.

\begin{lemma}\label{lem: coercivity TNDL}
  For any frequency $\omega\geq0$ and parameter $\beta>0$, there exist a positive constant $C>0$ and a parameter $\theta>0$, possibly depending on $\Omega_s$, $\mu$, $\epsilon$ and $\kappa$, and a compact linear operator $\mathcal{K}:\mathbf{H}\bra{\mathbf{curl}, \Omega_s}\times H^1(\Omega_s)\rightarrow \mathbf{H}\bra{\mathbf{curl}, \Omega_s}\times H^1(\Omega_s)$ such that
  \begin{multline*}
    \mathfrak{Re}\Bigg(\Big\langle {\color{purple}-\mathbb{A}^{DN}_{\kappa}}\colvec{2}{\gamma_t^-\mathbf{U}}{-\gamma^-\bra{P}},\colvec{2}{\gamma^{-}_t\Xi_1\vec{\overline{\mathbf{U}}}}{\gamma^{-} \Xi_2\vec{\overline{\mathbf{U}}}}\Big\rangle\\ +\Big\langle\mathcal{K}\colvec{2}{\gamma_t^-\mathbf{U}}{-\gamma^-\bra{P}},\colvec{2}{\gamma^{-}_t\Xi_1\vec{\overline{\mathbf{U}}}}{\gamma^{-} \Xi_2\vec{\overline{\mathbf{U}}}}\Big\rangle  \Bigg)
    \geq C\norm{\colvec{2}{\gamma_t^-\mathbf{U}}{\gamma^-(P)}}^2_{\mathcal{H}_D(\Omega_s)}
  \end{multline*}
  for all $\vec{\mathbf{U}}:=\bra{\mathbf{U}\,\,P}^{\top}\in\mathbf{H}\bra{\mathbf{curl}, \Omega_s}\times H^1(\Omega_s)$.
\end{lemma}
\begin{proof}
The jump condition \eqref{double layer jump} yield $\{\mathcal{T}_{N}\}\cdot\mathcal{DL}_{\kappa}=\mathcal{T}_{N}\cdot\mathcal{DL}_{\kappa}$. We deduce from \cite[Sec. 6.4]{claeys2017first} that,
\begin{align}
&\Big\langle -\mathcal{T}_{N}\cdot\mathcal{DL}_{\kappa}\colvec{2}{\gamma_t^-\mathbf{U}}{-\gamma^-\bra{P}},\colvec{2}{\gamma^{-}_t\Xi_1\vec{\mathbf{U}}}{\gamma^{-} \Xi_2\vec{\mathbf{U}}}\Big\rangle \nonumber\\
&\hat{=} \bra{\text{div}_{\Gamma}\bra{\mathbf{n}\times\gamma_t^-\mathbf{U})},\text{div}_{\Gamma}\bra{\mathbf{n}\times\gamma^{-}_t\Xi_1\vec{\mathbf{U}})}}_{-1/2}\nonumber\\
&\qquad- \kappa^2\bra{\mathbf{n}\times\gamma_t^-\mathbf{U},\mathbf{n}\times\gamma^{-}_t\Xi_1\vec{\mathbf{U}}}_{-1/2}
+\bra{\mathbf{n}\times\gamma_t^-\mathbf{U},\mathbf{curl}_{\Gamma}\bra{\gamma^{-} \Xi_2\vec{\mathbf{U}}}}_{-1/2}\nonumber\\
&\qquad-\bra{\mathbf{n}\times\gamma^{-}_t\Xi_1\vec{\mathbf{U}},\mathbf{curl}_{\Gamma}\bra{\gamma^-\bra{P}}}_{-1/2}\nonumber\\
&=\bra{\text{div}_{\Gamma}\bra{\gamma_{\tau}^-\mathbf{U})},\text{div}_{\Gamma}\bra{\gamma^{-}_{\tau}\Xi_1\vec{\mathbf{U}})}}_{-1/2}
- \kappa^2\bra{\gamma_{\tau}^-\mathbf{U},\gamma^{-}_{\tau}\Xi_1\vec{\mathbf{U}}}_{-1/2}\nonumber\\
&\qquad-\bra{\gamma_{\tau}^-\mathbf{U},\mathbf{curl}_{\Gamma}\bra{\gamma^{-} \Xi_2\vec{\mathbf{U}}}}_{-1/2}
+\bra{\gamma^{-}_{\tau}\Xi_1\vec{\mathbf{U}},\mathbf{curl}_{\Gamma}\bra{\gamma^-\bra{P}}}_{-1/2}\label{eq: TNDL up to compact}
\end{align}

We consider each component of the isomorphim $\Xi$ in turn. Since $\mathsf{Z}\bra{\mathbf{U}}\in\mathbf{H}^1(\Omega_s)$ \cite[Lem. 3.5]{amrouche1998vector} and $\gamma_t\mathbf{H}^1(\Omega_s)$ is compactly embedded in $\mathbf{L}^2_t(\Gamma)$ \cite[Lem. 3.2]{hiptmair2003coupling}, the continous mapping $\gamma_{\tau}\circ \mathsf{Z}:\mathbf{H}\bra{\mathbf{curl},\Omega_s}\rightarrow\mathbf{H}_R^{1/2}\bra{\Omega_s}$ is compact. Therefore,
\begin{align}
	\gamma^-_{\tau}\Xi_1\bra{\vec{\mathbf{U}}} &= \gamma^-_\tau\mathbf{U}^{\perp}-\gamma^-_\tau\mathbf{U}^0+\beta\gamma^-_{\tau}\nabla P\nonumber\\
&\,\,\,\hat{=}\, \mathsf{Z}^{\Gamma}\bra{\gamma^-_{\tau}\mathbf{U}} - \bra{\id-\mathsf{Z}^{\Gamma}}\gamma^-_{\tau}\mathbf{U}+\beta\,\mathbf{curl}_{\Gamma}\bra{\gamma^-{P}}.\label{eq: tau trace using commutative diagram}
\end{align}
Let's introduce expression \eqref{eq: tau trace using commutative diagram} in the various terms of \eqref{eq: TNDL up to compact} involving $\Xi_1(\vec{\mathbf{U}})$. We find that
		\begin{align*}
		\Big(\text{div}_{\Gamma}\left(\gamma_{\tau}^-\mathbf{U})\right),&\text{div}_{\Gamma}\bra{\gamma^{-}_{\tau}\Xi_1\vec{\mathbf{U}})}\Big)_{-1/2}\\
		&\hat{=}
		\bra{\text{div}_{\Gamma}\bra{\gamma_{\tau}\mathbf{U}},\text{div}_{\Gamma}\bra{\mathsf{Z}^{\Gamma}\bra{\gamma^-_{\tau}\mathbf{U}}}}_{-1/2}\\
		&\qquad\qquad-\bra{\text{div}_{\Gamma}\bra{\gamma_{\tau}\mathbf{U}},\text{div}_{\Gamma}\bra{\bra{\id-\mathsf{Z}^{\Gamma}}\gamma^-_{\tau}\mathbf{U}}}_{-1/2}\\
		&\qquad\qquad+\beta\bra{\text{div}_{\Gamma}\bra{\gamma_{\tau}\mathbf{U}},\text{div}_{\Gamma}\bra{\mathbf{curl}_{\Gamma}\bra{\gamma^-{P}}}}_{-1/2}\\
		&=	\bra{\text{div}_{\Gamma}\bra{\gamma^-_{\tau}\mathbf{U}},\text{div}_{\Gamma}\bra{\gamma^-_{\tau}\mathbf{U}}}_{-1/2}.
	\end{align*}
	Similarly,
	\begin{multline*}
		- \kappa^2\bra{\gamma_{\tau}^-\mathbf{U},\gamma^{-}_{\tau}\Xi_1\vec{\mathbf{U}}}_{-1/2}  
		\hat{=} \, \,\kappa^2\bra{\bra{\id-\mathsf{Z}^{\Gamma}}\gamma_{\tau}^-\mathbf{U},\bra{\id-\mathsf{Z}^{\Gamma}}\gamma^-_{\tau}\mathbf{U}}_{-1/2}\\
		-\beta\kappa^2\bra{\bra{\id-\mathsf{Z}^{\Gamma}}\gamma_{\tau}^-\mathbf{U},\mathbf{curl}_{\Gamma}\bra{\gamma^-{P}}}_{-1/2}
	\end{multline*}
	and
	\begin{multline*}
		\bra{\gamma^{-}_{\tau}\Xi_1\vec{\mathbf{U}},\mathbf{curl}_{\Gamma}\bra{\gamma^-\bra{P}}}_{-1/2}
		\hat{=}
		-\bra{\bra{\id-\mathsf{Z}^{\Gamma}}\gamma^-_{\tau}\mathbf{U},\mathbf{curl}_{\Gamma}\bra{\gamma^-\bra{P}}}_{-1/2} \\
		\qquad+\beta\bra{\mathbf{curl}_{\Gamma}\bra{\gamma^-{P}},\mathbf{curl}_{\Gamma}\bra{\gamma^-\bra{P}}}_{-1/2}. 
	\end{multline*}
We now want to evaluate the terms involving $\Xi_2(\vec{\mathbf{U}})$. We introduce
\begin{equation*}
\mathbf{curl}_{\Gamma}\bra{\gamma^{-} \Xi_2\vec{\mathbf{U}}}=-\theta\gamma^-_{\tau}\nabla\bra{\mathsf{S}(\mathbf{U}^0)+\mathbf{mean}(P)}=-\theta\bra{\id-\mathsf{Z}^{\Gamma}}\gamma^-_{\tau}\mathbf{U},
\end{equation*}
in \eqref{eq: TNDL up to compact} to obtain
\begin{equation*}
-\bra{\gamma_{\tau}^-\mathbf{U},\mathbf{curl}_{\Gamma}\bra{\gamma^{-} \Xi_2\vec{\mathbf{U}}}}_{-1/2} =\theta\bra{\bra{\id-\mathsf{Z}^{\Gamma}}\gamma^-_{\tau}\mathbf{U},\bra{\id - \mathsf{Z}^{\Gamma}}\gamma_{\tau}\mathbf{U}}_{-1/2}
\end{equation*}
Using Young's inequality twice with $\delta>0$,
		\begin{align*}
	&\mathfrak{Re}\bra{\Big\langle -\{\mathcal{T}_{N}\}\cdot\mathcal{DL}_{\kappa}\colvec{2}{\gamma_t^-\mathbf{U}}{-\gamma^-\bra{P}},\colvec{2}{\gamma^{-}_t\Xi_1\vec{\mathbf{U}}}{\gamma^{-} \Xi_2\vec{\mathbf{U}}}\Big\rangle}\\
	&\hat{=} \,\norm{\text{div}_{\Gamma}\bra{\gamma^-_{\tau}\mathbf{U}}}^2_{-1/2} + \bra{\mathfrak{Re}\bra{\kappa^2}+\theta}\norm{\bra{\id-\mathsf{Z}^{\Gamma}}\gamma_{\tau}^-\mathbf{U}}^2_{-1/2}\\
	&\qquad+ \beta\norm{\mathbf{curl}_{\Gamma}\bra{\gamma^-\bra{P}}}^2
	-\bra{\bra{\id-\mathsf{Z}^{\Gamma}}\gamma^-_{\tau}\mathbf{U},\mathbf{curl}_{\Gamma}\bra{\gamma^-\bra{P}}}_{-1/2}\\ 
	&\qquad-\beta\,\mathfrak{Re}\bra{\kappa^2}\bra{\bra{\id-\mathsf{Z}^{\Gamma}}\gamma_{\tau}^-\mathbf{U},\mathbf{curl}_{\Gamma}\bra{\gamma^-{P}}}_{-1/2}\\
	&\geq \,\,\norm{\text{div}_{\Gamma}\bra{\gamma^-_{\tau}\mathbf{U}}}^2_{-1/2}+ \bra{\beta - \frac{1}{\delta}\bra{1+\beta\,\mathfrak{Re}\bra{\kappa^2}}}\norm{\mathbf{curl}_{\Gamma}\bra{\gamma^-\bra{P}}}^2\\
	&\qquad+\bra{\mathfrak{Re}\bra{\kappa^2}+\theta-\delta\,\bra{1 +\beta\,\mathfrak{Re}\bra{\kappa^2}}}\norm{\bra{\id-\mathsf{Z}^{\Gamma}}\gamma_{\tau}^-\mathbf{U}}^2_{-1/2}. \\
\end{align*}
The operator $\mathbf{curl}_{\Gamma}:H^1_*(\Omega_s)\rightarrow \mathbf{H}^{-1/2}\bra{\text{div}_{\Gamma},\Gamma}$ is a continuous injection \cite[Lem. 6.4]{claeys2017first}. It is thus bounded below. Since the mean operator has finite rank, it is compact. Therefore, for any $\beta>0$, choose $\delta>0$ large enough, then $\theta>0$ accordingly large, and the desired inequality follows by equivalence of norms. 
\end{proof}

In the next lemma, we prove coercivity of the lower diagonal block of the coupling operator $\mathbb{G}_{\kappa}$.
\begin{lemma}\label{lem: coercivity of TDSL}
  For any frequency $\omega\geq0$, there exist a compact linear operator $\mathcal{K}:\mathcal{H}_N\rightarrow \mathcal{H}_D$, a positive constants $C>0$ and parameters $\tau>0$ and $\lambda>0$, possibly depending on $\Omega_s$, $\mu$, $\epsilon$ and $\kappa$, such that
  \begin{equation*}
    \mathfrak{Re}\bra{\Big\langle{\color{magenta}\mathbb{A}^{ND}_{\kappa}}\bra{\vec{\mathbf{p}}}, \Xi^{\Gamma}\vec{\overline{\mathbf{p}}}
      \Big\rangle +\Big\langle\mathcal{K}\,\vec{\mathbf{p}},\vec{\overline{\mathbf{p}}}\Big\rangle} \geq C\norm{\vec{\mathbf{p}}}^2_{\mathcal{H}_N}
  \end{equation*}
  for all $\vec{\mathbf{p}}\in\mathcal{H}_N$.	In particular, for $\mathfrak{Re}\bra{k^2}\neq 0$, the inequality holds with $\tau = 1/\kappa^2$.
\end{lemma}
\begin{proof}
  The jump conditions \eqref{single layer jump} yield $\{\mathcal{T}_D\}\cdot\mathcal{SL}\bra{\vec{\mathbf{p}}}=\mathcal{T}_D\cdot\mathcal{SL}\bra{\vec{\mathbf{p}}}$. We deduce from \cite[Sec. 6.3]{claeys2017first} and the compact embedding of $\mathbf{X}\bra{\text{div}_{\Gamma},\Gamma}$ into $\mathbf{H}_R^{-1/2}(\Gamma)$ that
  \begin{align*}
    \Big\langle\mathcal{T}_D\cdot\mathcal{SL}\bra{\vec{\mathbf{p}}},\Xi^{\Gamma}\vec{\mathbf{p}}\Big\rangle &\hat{=}-\bra{\mathbf{p}^0,\Xi^{\Gamma}_1(\mathbf{p})}_{-1/2} -\bra{q,\text{div}_{\Gamma}\bra{\Xi^{\Gamma}_1(\mathbf{p})}}_{-1/2} \\ &\qquad-\bra{\text{div}_{\Gamma}(\mathbf{p}),\Xi^{\Gamma}_2\vec{\mathbf{p}}}_{-1/2} -\kappa^2\bra{q,\Xi^{\Gamma}_2\bra{\vec{\mathbf{p}}}}_{-1/2}\\
    &\hat{=}\bra{\mathbf{p}^0,\mathbf{p}^0}_{-1/2} -\bra{q,\text{div}_{\Gamma}(\mathbf{p}^{\perp})}_{-1/2} +\lambda\bra{q,\mathsf{Q}_*q}_{-1/2} \\
    &\qquad+\tau\bra{\text{div}_{\Gamma}(\mathbf{p}),\text{div}_{\Gamma}(\mathbf{p})}_{-1/2} + \tau\kappa^2\bra{q,\text{div}_{\Gamma}(\mathbf{p}^{\perp})}_{-1/2}.
  \end{align*}	
  When $\mathfrak{Re}\bra{\kappa^2} > 0$, setting $\tau=1/\kappa^2$ immediately yields the existence of a compact linear operator $\mathcal{K}:\mathcal{H}_N\rightarrow \mathcal{H}_D$ such that
  \begin{multline*}
    \Big\langle\mathcal{T}_D\cdot\mathcal{SL}\bra{\vec{\mathbf{p}}},\Xi^{\Gamma}\vec{\mathbf{p}}\Big\rangle 
    + \Big\langle \mathcal{K}\vec{\mathbf{p}},\Xi^{\Gamma}\vec{\mathbf{p}}\Big\rangle\\ \geq C\bra{ \norm{\text{div}_{\Gamma}\bra{\mathbf{p}}}^2_{-1/2} + \norm{\mathbf{p}^0}^2_{-1/2} + \norm{\mathsf{Q}_*q}^2_{-1/2}}.
  \end{multline*}
  When $\kappa^2 = 0$, the same inequality is obtained for any $\lambda>0$ by using Young's inequality as in the proof of Lemma \ref{lem: coercivity TNDL} and choosing $\tau$ large enough. The claimed inequality follows by equivalence of norms.
\end{proof}

Equipped with the previous three lemmas, we are now ready to prove Proposition \ref{prop: coercivity volume and TNDL}.

\begin{proof}[Proof of Proposition \ref{prop: coercivity volume and TNDL}]
  For any parameters $\beta>0$ and $\lambda>0$, the choices of $\delta$ and $\theta$ in the proofs of Lemma \ref{lem: coercivity B} and Lemma \ref{lem: coercivity TNDL} are not mutually exclusive. The choice of $\tau$ in Lemma \ref{lem: coercivity of TDSL} is independent of the choice of $\theta$.
\end{proof}

\subsection{Compactness of the off-diagonal blocks}
Finally, The off-diagonal blocks remain to be considered. We will show that, up to compact perturbations, a suitable choice of parameters in the isomorphisms $\Xi$ and $\Xi^{\Gamma}$ of the test space leads to a skew-symmetric pattern in $\mathbb{G}_{\kappa}$. In other words, up to compact terms, the volume and boundary parts of the system decouples over the space decompositions introduced in Section \ref{sec: space decomposition}.

\begin{proof}[Proof of Proposition \ref{prop: compactness off-diagonal blocks}]
  The isomorphisms $\Xi$ and $\Xi^{\Gamma}$ were designed so that favorable cancellations occur in evaluating the left hand side of \eqref{eq: K kappa compact coercivity}. 
  
  From the jump properties \eqref{single layer jump}, we have $\{\mathcal{T}_N\}\mathcal{SL}_{\kappa} = \mathcal{T}_N^-\mathcal{SL}_{\kappa} -(1/2)\id$. Therefore, as in \eqref{eq: average single layer}, we evaluate
  \begin{align}
    &\Big\langle{\color{olive}\bra{\mathbb{P}_{\kappa}^+}_{22}\vec{\mathbf{p}}},\colvec{2}{\gamma^-_t\Xi_1\vec{\overline{\mathbf{U}}}}{\gamma^-\Xi_2\vec{\overline{\mathbf{U}}}} \Big\rangle\nonumber\\
    &=\Big\langle\bra{-\{\mathcal{T}_N\}\cdot\mathcal{SL}_{\kappa}+\frac{1}{2}\id}\vec{\mathbf{p}},\colvec{2}{\gamma^-_t\Xi_1\vec{\overline{\mathbf{U}}}}{\gamma^-\Xi_2\vec{\overline{\mathbf{U}}}}\Big\rangle\nonumber\\
    &= \Big\langle-\mathcal{T}_N^-\cdot\mathcal{SL}_{\kappa}\bra{\vec{\mathbf{p}}},\colvec{2}{\gamma^-_t\Xi_1\vec{\overline{\mathbf{U}}}}{\gamma^-\Xi_2\vec{\overline{\mathbf{U}}}}\Big\rangle + \Big\langle \vec{\mathbf{p}},\colvec{2}{\gamma^-_t\Xi_1\vec{\overline{\mathbf{U}}}}{\gamma^-\Xi_2\vec{\overline{\mathbf{U}}}}\Big\rangle\nonumber\\
    &= \langle \gamma^-_R\bm{\Psi}_{\kappa}\bra{\mathbf{p}},\gamma_t^-\Xi_1\vec{\overline{\mathbf{U}}}\rangle_{\tau}
    - \langle\gamma_n^-\nabla\psi_{\tilde{\kappa}}\bra{q},\gamma^-\Xi_2\vec{\overline{\mathbf{U}}} \rangle_{\Gamma}+ \langle \gamma_n^-\bm{\Psi}_{\kappa}\bra{\mathbf{p}}, \gamma^-\Xi_2\vec{\overline{\mathbf{U}}}\rangle_{\Gamma} \nonumber\\
    &\qquad
    +\langle\gamma^-_n\nabla\tilde{\psi}_{\kappa}\bra{\text{div}_{\Gamma}\mathbf{p}},\gamma^-\Xi_2\vec{\overline{\mathbf{U}}}\rangle_{\Gamma}
   + \langle \mathbf{p},\gamma_t^-\Xi_1\vec{\overline{\mathbf{U}}}\rangle_{\tau}
    + \langle q, \gamma^-\Xi_2\vec{\overline{\mathbf{U}}}\rangle_{\Gamma}\nonumber\\
&\hat{=}\,
    {\color{red}\langle \gamma^-_R\bm{\Psi}_{\kappa}\bra{\mathbf{p}^0},\gamma_t\overline{\mathbf{U}}^{\perp}\rangle_{\tau}} 
    {\color{blue} 
    - \langle \gamma^-_R\bm{\Psi}_{\kappa}\bra{\mathbf{p}^0},\gamma_t\overline{\mathbf{U}^0}\rangle_{\tau}} 
   {\color{blue} + \beta\,\langle \gamma^-_R\bm{\Psi}_{\kappa}\bra{\mathbf{p}^0},\gamma_t\nabla\overline{P}\rangle_{\tau}}\nonumber\\
    &\qquad+{\color{blue}\langle\gamma^-_R\bm{\Psi}_{\kappa}\bra{\mathbf{p}^{\perp}},\gamma_t\overline{\mathbf{U}}^{\perp}\rangle_{\tau}}
    {\color{red} - \langle \gamma^-_R\bm{\Psi}_{\kappa}\bra{\mathbf{p}^{\perp}},\gamma_t\overline{\mathbf{U}^0}\rangle_{\tau}} \nonumber\\
    &\qquad+ \beta\,\langle \gamma^-_R\bm{\Psi}_{\kappa}\bra{\mathbf{p}^{\perp}},\gamma_t\nabla\overline{P}\rangle_{\tau}
    +\theta\,\langle\gamma^-_n\nabla\psi_{\tilde{\kappa}}\bra{q}, \gamma^-\mathsf{S}\bra{\overline{\mathbf{U}}^0}\rangle_{\Gamma}\nonumber\\
    &\qquad-\theta\,\langle \gamma^-_n\bm{\Psi}_{\kappa}\bra{\mathbf{p}},\gamma^- \mathsf{S}\bra{\overline{\mathbf{U}}^0}\rangle_{\Gamma}
    -\langle \gamma^-_n\nabla\tilde{\psi}_{\kappa}\bra{\text{div}_{\Gamma}\mathbf{p}},\theta\,\gamma^- \mathsf{S}\bra{\overline{\mathbf{U}}^0}\rangle_{\Gamma} \nonumber\\
    &\qquad{\color{red} + \langle\mathbf{p}^0,\gamma_t^-\overline{\mathbf{U}}^{\perp} \rangle_{\tau}}
    {\color{blue} + \langle\mathbf{p}^\perp,\gamma_t^-\overline{\mathbf{U}}^{\perp} \rangle_{\tau}}
{\color{blue} - \langle\mathbf{p}^0,\gamma_t^-\overline{\mathbf{U}}^{0} \rangle_{\tau}}
    {\color{red} - \langle\mathbf{p}^{\perp},\gamma_t^-\overline{\mathbf{U}}^{0} \rangle_{\tau}}\nonumber\\
    &\qquad{\color{blue} + \beta \,\langle \mathbf{p}^0,\gamma_t^-\nabla\overline{P}\rangle_{\tau}}
    + \beta \,\langle \mathbf{p}^{\perp},\gamma_t^-\nabla\overline{P}\rangle_{\tau}
    - \theta\,\langle q,\gamma^- \mathsf{S}\bra{\mathbf{U}^0}\rangle_{\Gamma},\label{eq: off-diagonal first SL}
  \end{align}
  where we have used that the finite rank of the mean operator implies compactness.
  
  Similarly, using Proposition \ref{prop: adjointness}, we find
  \begin{align}
    &\Big\langle{\color{teal}\bra{\mathbb{P}_{\kappa}^-}_{11}}\colvec{2}{\gamma_t^-\mathbf{U}}{-\gamma^-\bra{P}}, \Xi^{\Gamma}\vec{\overline{\mathbf{p}}}\Big\rangle = \Big\langle \colvec{2}{\gamma_t^-\mathbf{U}}{-\gamma^-\bra{P}},{\color{olive}\bra{\mathbb{P}_{\kappa}^+}_{22}}\Xi^{\Gamma}\vec{\overline{\mathbf{p}}}\Big\rangle\nonumber\\
    &\hat{=}\, {\color{red}\langle\gamma^-_R\bm{\Psi}_{\kappa}\bra{\overline{\mathbf{p}}^{\perp}},\gamma_t\mathbf{U}^0\rangle_{\tau}}
    {\color{blue}-\langle\gamma^-_R\bm{\Psi}_{\kappa}\bra{\overline{\mathbf{p}}^{0}},\gamma_t\mathbf{U}^{\perp}\rangle_{\tau}}\nonumber\\ 
    &\qquad-\lambda\,\langle\gamma_R^-\bm{\Psi}_{\kappa}\bra{\bra{\text{div}_{\Gamma}}^{\dag}Q_*\overline{q}},\gamma_t^-\mathbf{U}^0 \rangle_{\tau}
    +{\color{blue}\langle\gamma^-_R\bm{\Psi}_{\kappa}\bra{\overline{\mathbf{p}}^{0}},\gamma_t\mathbf{U}^{\perp}\rangle_{\tau}}\nonumber\\
    &\qquad{\color{red}-\langle\gamma^-_R\bm{\Psi}_{\kappa}\bra{\overline{\mathbf{p}}^{0}},\gamma_t\mathbf{U}^{\perp}\rangle_{\tau}} 
    {\color{blue}-\lambda\,\langle\gamma_R^-\bm{\Psi}_{\kappa}\bra{\bra{\text{div}_{\Gamma}}^{\dag}Q_*\overline{q}},\gamma_t^-\mathbf{U}^{\perp} \rangle_{\tau}}\nonumber\\
    &\qquad-\tau\,\langle \gamma_n^-\nabla\psi_{\tilde{\kappa}}\bra{\text{div}_{\Gamma}\overline{\mathbf{p}}^{\perp}},\gamma^-P\rangle
    -\langle\gamma^-_n\bm{\Psi}_{\kappa}\bra{\overline{\mathbf{p}}^{\perp}},\gamma^-P\rangle_{\Gamma}\quad...\nonumber\\
    \end{align}
    \begin{align}
    ...&\qquad+\langle\gamma^-_n\bm{\Psi}_{\kappa}\bra{\overline{\mathbf{p}}^{0}},\gamma^-P\rangle_{\Gamma}
    +\lambda\,\langle\gamma^-_n\bm{\Psi}_{\kappa}\bra{\bra{\text{div}}^{\dag}Q_*\overline{q}},\gamma^-P\rangle_{\Gamma}\nonumber\\
    &\qquad-\langle\gamma_n^-\nabla\tilde{\psi}_{\kappa}\bra{\text{div}_{\Gamma}\overline{\mathbf{p}}^{\perp}},\gamma^-P\rangle_{\Gamma}
    + \lambda\,\langle\gamma_n^-\nabla\tilde{\psi}_{\kappa}\bra{Q_*\overline{q}},\gamma^-P\rangle_{\Gamma}\nonumber\\
    &\qquad{\color{red}+\langle\gamma^-_t\mathbf{U}^0,\overline{\mathbf{p}}^{\perp}\rangle_{\tau}}
    {\color{blue}+\langle\gamma^-_t\mathbf{U}^{\perp},\overline{\mathbf{p}}^{\perp}\rangle_{\tau}}
    {\color{red}-\langle\gamma^-_t\mathbf{U}^{\perp},\overline{\mathbf{p}}^{0}\rangle_{\tau}}\nonumber\\
    &\qquad{\color{blue}-\langle\gamma^-_t\mathbf{U}^0,\overline{\mathbf{p}}^{0}\rangle_{\tau}}
    {\color{blue}- \lambda\,\langle\gamma^-_t\mathbf{U}^{\perp},\bra{\text{div}_{\Gamma}}^{\dag}Q_*\overline{q}\rangle_{\Gamma}}\nonumber\\
    &\qquad- \lambda\,\langle\gamma^-_t\mathbf{U}^{0},\bra{\text{div}_{\Gamma}}^{\dag}Q_*\overline{q}\rangle
    +\tau\,\langle \gamma^- P,\text{div}_{\Gamma}\bra{\overline{\mathbf{p}}^{\perp}}\rangle_{\Gamma}.\label{eq: off-diagonal second SL}
  \end{align}
  Many terms in these equations can be combined and asserted compact by \eqref{eq: old lem 4.6 1} and \eqref{eq: old lem 4.6 2}. They are indicated in {\color{blue}blue}. When summing the real
  parts of \eqref{eq: off-diagonal first SL} and \eqref{eq: off-diagonal second SL}, the
  terms in {\color{red}red} cancel. Relying on \eqref{eq: compact nu 0 phi} to \eqref{eq: compact nu 0 phi tilde}, some terms amount to compact perturbations so that we may replace $\kappa$
  and $\tilde{\kappa}$ by zero in those instances. We have arrived at the following identity:
  \begin{align*}
    \begin{split}
      \mathfrak{Re}&\bra{\Big\langle \bra{\mathbb{G}_{\kappa}-\diag\bra{\mathbb{G}_{\kappa}}}\colvec{2}{\vec{\mathbf{U}}}{\vec{\mathbf{p}}},\colvec{2}{\Xi\,\vec{\overline{\mathbf{U}}}}{\Xi^{\Gamma}\vec{\overline{\mathbf{p}}}}\Big\rangle}\\
      &\hat{=}\,\,\mathfrak{Re}\Bigg(
      \beta\,\langle \gamma^-_R\bm{\Psi}_{0}\bra{\mathbf{p}^{\perp}},\gamma_t\nabla\overline{P}\rangle_{\tau}
      +\theta\,\langle\gamma^-_n\nabla\psi_{0}\bra{q}, \gamma^-\mathsf{S}\bra{\overline{\mathbf{U}}^0}\rangle_{\Gamma}\\
      &\quad{\color[rgb]{0,0.5,0.15}-\theta\,\langle \gamma^-_n\bm{\Psi}_0\bra{\mathbf{p}},\gamma^- \mathsf{S}\bra{\overline{\mathbf{U}}^0}\rangle_{\Gamma}} 	
      {\color[rgb]{0.8,0.33,0}+ \beta \,\langle \mathbf{p}^{\perp},\gamma_t^-\nabla\overline{P}\rangle_{\tau}
        - \theta\,\langle q,\gamma^- \mathsf{S}\bra{\mathbf{U}^0}\rangle_{\Gamma}}\\
      &\quad- \lambda\,\langle\gamma_R^-\bm{\Psi}_{0}\bra{\bra{\text{div}_{\Gamma}}^{\dag}Q_*\overline{q}},\gamma_t^-\mathbf{U}^0 \rangle_{\tau}
      -\tau\,\langle \gamma_n^-\nabla\psi_{0}\bra{\text{div}_{\Gamma}\overline{\mathbf{p}}^{\perp}},\gamma^-P\rangle_{\Gamma}\\
      &\quad{\color[rgb]{0,0.5,0.15}-\langle\gamma^-_n\bm{\Psi}_0\bra{\overline{\mathbf{p}}^{\perp}},\gamma^-P\rangle_{\Gamma}
        +\langle\gamma^-_n\bm{\Psi}_0\bra{\overline{\mathbf{p}}^{0}},\gamma^-P\rangle_{\Gamma}}\\
      &	\quad{\color[rgb]{0,0.5,0.15}+\lambda\,\langle\gamma^-_n\bm{\Psi}_0\bra{\bra{\text{div}}^{\dag}Q_*\overline{q}},\gamma^-P\rangle_{\Gamma}}\\
      &\quad{\color[rgb]{0.8,0.33,0}- \lambda\,\langle\gamma^-_t\mathbf{U}^{0},\bra{\text{div}_{\Gamma}}^{\dag}Q_*\overline{q}\rangle_{\tau}
        +\tau\,\langle \gamma^- P,\text{div}_{\Gamma}\bra{\overline{\mathbf{p}}^{\perp}}\rangle_{\Gamma}}\Bigg).
    \end{split}
  \end{align*}
  We claim that the terms colored in {\color[rgb]{0,0.5,0.15}green} are compact. Indeed, the integral identities of Section \ref{sec: classical traces} together with equality \eqref{scalar and vector single layer potential identity} yield
  \begin{align*}
    &\langle \gamma^-_n\bm{\Psi}_0\bra{\mathbf{p}},\gamma^-
    \mathsf{S}\bra{\overline{\mathbf{U}}^0}\rangle_{\Gamma}\nonumber
    \\
    &\qquad\qquad\qquad\qquad\leq \bra{\norm{\psi_0\bra{\text{div}_{\Gamma}\mathbf{p}}}_{L^2(\Omega_s)}+\norm{\bm{\Psi}_0\bra{\mathbf{p}}}_{\mathbf{L}^2(\Omega_s)}}\norm{\overline{\mathbf{U}}^0}_{\mathbf{L}^2(\Omega_s)},\\
    &\langle \gamma_n^-\bm{\Psi}_0\bra{\overline{\mathbf{p}}},\gamma^-P\rangle_{\Gamma}\nonumber\\
    &\qquad\qquad\qquad\qquad\leq \bra{\norm{\psi_0\bra{\text{div}_{\Gamma}\overline{\mathbf{p}}}}_{L^2(\Omega_s)}+\norm{\bm{\Psi}_0\bra{\overline{\mathbf{p}}}}_{\mathbf{L}^2(\Omega_s)}}\norm{P}_{H^1(\Omega_s)}\\
    &\langle\gamma^-_n\bm{\Psi}_0\bra{\bra{\text{div}}^{\dag}Q_*\overline{q}},\gamma^-P\rangle_{\Gamma},\nonumber\\
    &\qquad\qquad\qquad\qquad\leq \bra{\norm{\psi_0\bra{Q_*q}}_{L^2(\Omega_s)}+\norm{\bm{\Psi}_0\bra{\text{div}_{\Gamma}\overline{\mathbf{p}}}}_{\mathbf{L}^2(\Omega_s)}}\norm{P}_{H^1(\Omega_s)}.
  \end{align*}
  Since $\psi_{0}:H^{-1/2}(\Gamma)\rightarrow H^1(\Omega_s)$ and
  $\bm{\Psi}_{0}:\mathbf{H}^{-1/2}(\Gamma)\rightarrow \mathbf{H}^1(\Omega_s)$ are
  continuous, compactness is guaranteed by Rellich's Theorem.
  
  To go further, we need to settle for a choice of parameters in the volume and boundary
  isomorphisms. Choose $\tau$ to satisfy the requirements of Lemma \ref{lem: coercivity of
    TDSL}, then set $\beta = \tau$. We are still free to let $\theta$ satisfy both Lemma
  \ref{lem: coercivity B} and Lemma \ref{lem: coercivity TNDL}, and then choose
  $\lambda = \theta$.
  
  Under this choice of parameters, the terms in {\color[rgb]{0.8,0.33,0}orange} vanish, because we have $
  \langle\mathbf{p}^{\perp},\gamma_t^-\nabla\overline{P}\rangle_{\tau} = \langle\mathbf{p}^{\perp},\nabla_{\Gamma}\gamma^-\overline{P}\rangle_{\tau}=-\langle\text{div}_{\Gamma}\bra{\mathbf{p}^{\perp}},\gamma^-\overline{P}\rangle_{\Gamma} 
  $, and similarly
  \begin{align*}
    \langle \gamma^-_t\mathbf{U}^0,\bra{\text{div}_{\Gamma}}^{\dag}Q_*\overline{q}\rangle_{\tau} &=\langle \gamma^-_t\nabla\nonumber \mathsf{S}\bra{\mathbf{U}^0},\bra{\text{div}_{\Gamma}}^{\dag}Q_*\overline{q}\rangle_{\tau}\\
    &=-\langle \gamma^-\mathsf{S}\bra{\mathbf{U}^0},Q_*\overline{q}\rangle_{\Gamma}.
  \end{align*}
  Finally, relying on \eqref{scalar helmholtz for scalar single layer}, \eqref{vector
    helmholtz for vector single layer} and \eqref{scalar and vector single layer potential
    identity} once more, we observe that
  \begin{multline*}
    \langle\gamma_R^-\bm{\Psi}_0\bra{\mathbf{p}^{\perp}},\gamma_t^-\nabla \overline{P}\rangle_{\tau} =\bra{\mathbf{curl}\,\mathbf{curl}\,\bm{\Psi}_0\bra{\mathbf{p}^{\perp}},\nabla P}_{\Omega_s}\\
    =\bra{\nabla\psi_0\bra{\text{div}_{\Gamma}\mathbf{p}^{\perp}},\nabla P}_{\Omega_s}
    = \langle\gamma^-_n\nabla\psi_0\bra{\text{div}_{\Gamma}\mathbf{p}^{\perp}},\gamma^-\overline{P}\rangle_{\Gamma}.
  \end{multline*}
  A similar derivation shows that 
  \begin{gather*}
    \langle\gamma^-_n\nabla\psi_{0}\bra{q}, \gamma^-\mathsf{S}\bra{\overline{\mathbf{U}}^0}\rangle_{\Gamma}\,\hat{=}\, \langle\gamma_R^-\bm{\Psi}_{0}\bra{\bra{\text{div}_{\Gamma}}^{\dag}Q_*\overline{q}},\gamma_t^-\mathbf{U}^0 \rangle_{\tau}.
  \end{gather*}
  We conclude that for such a choice of parameters,
  \begin{gather*}
    \mathfrak{Re}\bra{\Big\langle \bra{\mathbb{G}_{\kappa}-\diag\bra{\mathbb{G}_{\kappa}}}\colvec{2}{\vec{\mathbf{U}}}{\vec{\mathbf{p}}},\colvec{2}{\Xi\,\vec{\overline{\mathbf{U}}}}{\Xi^{\Gamma}\vec{\overline{\mathbf{p}}}}\Big\rangle}
    \,\hat{=}\, 0,
  \end{gather*}
  which concludes the proof of this proposition.
\end{proof}

\section{Conclusion}

In section \ref{sec: coupled problem} we have proposed a system of equations coupling the
\emph{mixed formulation} of the variational form of the Hodge-Helmholtz and Hodge-Laplace
equation with \emph{first-kind} boundary integral equations. Well-posedness of the coupled
problem was obtained using a T-coercivity argument demonstrating that the operator
associated to the coupled variational problem was Fredholm of index 0. When
$\kappa^2\in\mathbb{C}$ avoids resonant frequencies, the operator's injectivity is
guaranteed, and thus stability of the problem is obtained along with the existence and
uniqueness of the solution. For such $\kappa^2$, Proposition \ref{prop: variational system
  solves transmission system} shows how solutions to the coupled variational problem are
in one-to-one correspondence with solutions of the transmission system. In principle, the
CFIE-type stabilization strategy applicable to transmission problems for the scalar Helmholtz operator \cite{hiptmair2005stable} or the electric wave equation \cite{MR2766824} could also be attempted here to get rid of the spurious
resonances haunting the coupled problem \eqref{Calderon coupled problem}, but such
developments lie outside the scope of this work.

The symmetrically coupled system \eqref{Calderon coupled problem} offers a variational
formulation of the transmission problem \eqref{eq: transmission problem} in well-known
energy spaces suited for discretization by finite and boundary elements. It is therefore a
promising starting point for Galerkin discretization.

\bibliographystyle{plain}
\bibliography{bibliography.bib}
\end{document}